\documentclass[oneside,english]{amsart}
\usepackage[T1]{fontenc}
\usepackage[latin9]{inputenc}
\usepackage{geometry}
\usepackage{amstext}
\usepackage{amsthm}
\usepackage{amssymb}
\usepackage{xcolor}
\usepackage{comment}
\usepackage{bm}

\makeatletter
\theoremstyle{plain}
\newtheorem{thm}{\protect\theoremname}
\theoremstyle{plain}
\newtheorem{lem}[thm]{\protect\lemmaname}

\newcommand{\dx}{\,\mathrm{d}}

\makeatother

\usepackage{babel}
\providecommand{\lemmaname}{Lemma}
\providecommand{\theoremname}{Theorem}

\begin{document}

\title{A Ces\`aro average for an additive problem with an arbitrary number of prime powers and squares}

\author{Marco Cantarini, Alessandro Gambini, Alessandro Zaccagnini}
\begin{abstract}
In this paper we extend and improve all the previous results known in literature about weighted average, with Ces\`aro weight, of
representations of an integer as sum of a positive arbitrary number of prime powers and a non-negative arbitrary number of squares. Our result includes all cases dealt with so far and allows us to obtain the best possible outcome using the chosen technique.
\end{abstract}

\maketitle
\emph{\today}\\

\subjclass 2010 Mathematics Subject Classification: {Primary 11P32; Secondary 44A10, 33C10}


\section{introduction}

The study of counting the number of the possible representation of a positive integer as sum of primes, prime powers and, in general,
sum of elements that belong to some fixed subset of $\mathbb{N}$ is classical in number theory and it has been highly developed in
recent years. Probably, the most popular problem in this context is the ternary and binary Goldbach conjecture, which states that every
odd number greater than $3$ is sum of three primes and every even number greater than $2$ is the sum of two primes, respectively. While
the ternary conjecture has been partially solved by Vinogradov \cite{Vin} and then definitively solved by Helfgott in a series of papers \cite{Helf3,Helf1,Helf2}, the binary conjecture is still open; only partial results were obtained
and, often, only conditional on the Riemann hypothesis (see, e.g.
the historical account about the Goldbach binary problem \cite{BhoHal}).
Given the difficulty of the problem, it was decided to tackle simplified
versions of it, such as, for example, considering the number
of representations on average and with suitable weights.  In this context
are inserted the works related to the study of averages, with Ces\`aro
weight, of functions that count the number representations of an integer
as the sum of elements that are in some fixed subset of $\mathbb{N}$. Similar averages of arithmetical 
functions are common in the literature: see, e.g., \cite{Ber}.
This approach was used in \cite{LanZac} for the binary Goldbach
problem: see Languasco's paper \cite{Languasco2016} for a thorough
introduction. The presence of a smooth weight allowed to obtain an asymptotic
formula with terms of decreasing orders of magnitude and depending
on the non-trivial zeros of the Riemann Zeta function; furthermore,
the weights allowed to obtain results independent of the Riemann hypothesis.
Since the Ces\`aro weights depend on a non-negative real parameter
$k$ and are equal to $1$ if $k=0$ (that is, for $k=0$ we have
a simple average without weights) it is important to obtain results
with the smallest possible $k$. During the last few years there have
been some improvements regarding the optimal $k$ in the case of the
Goldbach's problem: in \cite{LanZac} results hold for $k>1$, in
\cite{GolYan} (assuming Riemann hypothesis) and in \cite{Can1} (unconditionally)
for $k=1$ and in \cite{BruKacPer} for $k>0$. Unfortunately, for
case $k=0$, that is, without the Ces\`aro weights, it is not yet
possible to obtain in the same form or with the same quantity of terms
as in the other cases (see, for example, \cite{LanZac2},\cite{Can2}
and \cite{Pin}). Given the flexibility of the technique introduced
in \cite{LanZac}, the latter has been applied to other types of additive
problems (see \cite{Can3}\cite{Can5}\cite{LanZac3}\cite{LanZac4}\cite{LanZac5}).
In this paper, we prove a result which incorporates all the previous
results in the case of Ces\`aro averages and we show how the technique,
although very general and applicable to many problems, makes the lower
bound of $k$ worse as the number of primes and squares involved increases,
so as to confirm what has already been suggested by the lower
bound for $k$ obtained in \cite{Can3} and \cite{Can5}. 

\section{preliminary definitions and main theorem}

Let $d,h,N\in\mathbb{N},\,d>0$, $N\geq2$, $\mathbf{m}=\left(m_{1},\dots,m_{d}\right)\in\mathbb{N}^{d},\,\mathbf{r}:=\left(r_{1},\dots,r_{d}\right)\in\left(\mathbb{\mathbb{N}}^{+}\right)^{d},$ where $1\le r_1\le r_2\le \cdots\le r_d$,  $\mathbf{\mathbf{t}}:=\left(t_{1},\dots,t_{h}\right)\in\mathbb{N}^{h}$
and, in general, with bold letters, for example $\mathbf{f}$, we
will indicate some vector that belongs to $\mathbb{N}^{\alpha}$ or
$\left(\mathbb{N}^{+}\right)^{\alpha},$ for some positive integer
$\alpha.$ With the symbol $\left\Vert \cdot\right\Vert $ we will
indicate the usual Euclidean norm, with the symbol $\rho$, with or without subscripts, we will always indicate the non-trivial zeros of the Riemann zeta function and the series $\sum_{\rho}$ will always indicates the sum over all non trivial zeros of $\zeta(s)$, with or without subscripts. With $\bm{\rho}:=(\rho_{s_{1}},\dots,\rho_{s_{v}})$, where $s_{j},\, j=1,\dots,v$ belong to some subset of $\mathbb{N}^{+}$.

For every $\mathfrak{J}\subseteq\mathfrak{D}$ we will define the scalar product
\[
\tau\left(\bm{\Psi},\mathbf{r},\mathfrak{J}\right):=\sum_{j\in\mathfrak{J}}\frac{\Psi_{j}}{r_{j}}
\]
where $\bm{\Psi}=(\Psi_1,\ldots,\Psi_d)$. In most cases along the paper we will use $\bm{\Psi}=\bm{\rho}$; in addition, we will use the short definition $\tau\left(\bm{1},\mathbf{r},\mathfrak{J}\right)=\tau\left(\mathbf{r},\mathfrak{J}\right):=\sum_{j\in\mathfrak{J}}\frac{1}{r_{j}}$.

We will also indicate by $\sum_{\mathfrak{J}\subseteq\mathfrak{D}}$
the sums over all the possible subsets of of $\mathfrak{D}$. Taking
$n\in\mathbb{N}$, we set
\[
R_{d,h,\mathbf{r}}\left(n\right):=\sum_{m_{1}^{r_{1}}+\dots+m_{d}^{r_{d}}+t_{1}^{2}+\dots+t_{h}^{2}=n}\Lambda\left(m_{1}\right)\cdots\Lambda\left(m_{h}\right)
\]
where $\Lambda\left(m\right)$ is the usual von Mangoldt function.
We want to find an asymptotic formula, as $N\rightarrow+\infty$,
for 
\[
\sum_{n\leq N}R_{d,h,\mathbf{r}}\left(n\right)\frac{\left(N-n\right)^{k}}{\Gamma\left(k+1\right)}
\]
where $k>0$ is a real parameter and $\Gamma\left(x\right)$ is the
Euler Gamma function. 

Let $Z:=\left\{ s\in\mathbb{C},\,0\leq\mathrm{Re}\left(s\right)\leq1:\,\zeta\left(s\right)=0\right\}$  be the set of the non-trivial zeros of the Riemann Zeta function and let $\mathfrak{J}\subseteq\mathfrak{D}$. We will use the symbols 
$$\sum_{\bm{\rho}\in Z^{\left|\mathfrak{J}\right|}}=\sum_{\rho_{j_{1}}}\cdots\sum_{\rho_{j_{\left|\mathfrak{J}\right|}}}$$ 
and 
$$ \frac{1}{\textbf{r}}\Gamma\left(\frac{\bm{\rho}}{\textbf{r}}\right):=
\frac1{r_{j_1}}\Gamma\left(\frac{\rho_{j_{1}}}{r_{j_{1}}}\right)\cdots\frac1{r_{j_{_{\left|\mathfrak{J}\right|}}}}\Gamma\left(\frac{\rho_{j_{_{\left|\mathfrak{J}\right|}}}}{r_{j_{\left|\mathfrak{J}\right|}}}\right)$$ 
where $j_{\alpha}\in\mathfrak{J},\,\alpha=1,\dots,\left|\mathfrak{J}\right|$, every $\rho_{j_{\alpha}}\in Z$ and $r_{j_{\alpha}}$ is the $j_{\alpha}$-th coordinate of the fixed vector $\textbf{r}=\left(r_{1},\dots,r_{d}\right)$. 

In analogy to the previous definition, we will use the following symbol
$$ \frac{1}{\textbf{r}}\Gamma\left(\frac{1}{\textbf{r}}\right):=
\frac1{r_1}\Gamma\left(\frac{1}{r_{1}}\right)\cdots\frac1{r_{d}}\Gamma\left(\frac{1}{r_{d}}\right).$$ 

We introduce the following abbreviation for
the terms of the development:
\begin{align*}
M_{1}\left(N,k,d,h,\mathbf{r}\right)  :=&
\frac{1}{2^{h}}\sum_{\ell=0}^{h}\dbinom{h}{\ell}\frac{\pi^{\frac{\ell}{2}}\left(-1\right)^{h-\ell}N^{k+\tau\left(\mathbf{r},\mathfrak{D}\right)+\frac{\ell}{2}}}{\Gamma\left(k+1+\tau\left(\mathbf{r},\mathfrak{D}\right)+\frac{\ell}{2}\right)}
\frac{1}{\textbf{r}}\Gamma\left(\frac{1}{\textbf{r}}\right),
\\
 M_2\left(N,k,d,h,\mathbf{r}\right)  :=&
 \frac{N^{\frac{k+\tau\left(\mathbf{r},\mathfrak{D}\right)}{2}}}{\pi^{k+\tau\left(\mathbf{r},\mathfrak{D}\right)}}\sum_{\eta=0}^{h-1}\frac{\dbinom{h}{\eta}}{2^{\eta}}\sum_{\ell=0}^{\eta}\dbinom{\eta}{\ell}\left(-1\right)^{\eta-\ell}
\mathfrak{B}\big(\tau\left(\mathbf{r},\mathfrak{D}\right)\big)
\frac{1}{\textbf{r}}\Gamma\left(\frac{1}{\textbf{r}}\right),
\\
M_3\left(N,k,d,h,\mathbf{r}\right)  :=&
\frac{N^{k}\left(-1\right)^{d}}{2^{h}}\sum_{\ell=0}^{h}\dbinom{h}{\ell}\left(N\pi\right)^{\frac{\ell}{2}}\left(-1\right)^{h-\ell}\sum_{\bm{\rho}\in Z^{d}}
\frac1{\textbf{r}}\Gamma\left(\frac{\bm{\rho}}{\textbf{r}}\right)
\frac{N^{\tau\left(\bm{\rho},\mathbf{r},\mathfrak{D}\right)}}{\Gamma\left(k+1+\frac{\ell}{2}+\tau\left(\bm{\rho},\mathbf{r},\mathfrak{D}\right)\right)}.
\\
M_4\left(N,k,d,h,\mathbf{r}\right)  :=&
\frac{N^{k/2}\left(-1\right)^{d}}{\pi^{k}}\sum_{\eta=0}^{h-1}\frac{\dbinom{h}{\eta}}{2^{\eta}}\sum_{\ell=0}^{\eta}\dbinom{\eta}{\ell}\left(-1\right)^{\eta-\ell}
\sum_{\bm{\rho}\in Z^{d}}
\frac1{\textbf{r}}\Gamma\left(\frac{\bm{\rho}}{\textbf{r}}\right)
\frac{N^{\tau\left(\bm{\rho},\mathbf{r},\mathfrak{D}\right)/2}}{\pi^{\tau\left(\bm{\rho},\mathbf{r},\mathfrak{D}\right)}}
\mathfrak{B}\big(\tau\left(\bm{\rho},\mathbf{r},\mathfrak{D}\right)\big),
\\
M_5\left(N,k,d,h,\mathbf{r}\right)  :=&
\frac{N^{k/2}}{\pi^{k}}\sum_{\underset{{\scriptstyle \left|I\right|\geq1}}{I\subseteq\mathfrak{D}}}N^{\frac{\tau\left(\mathbf{r},I\right)}{2}}\left(-1\right)^{\left|\mathfrak{D}\setminus I\right|}\sum_{\eta=0}^{h}\frac{\dbinom{h}{\eta}}{2^{\eta}}\sum_{\ell=0}^{\eta}\dbinom{\eta}{\ell}\left(-1\right)^{\eta-\ell} 
\sum_{\bm{\rho}\in Z^{|\mathfrak{D}\setminus I|}}
\frac1{\textbf{r}}\Gamma\left(\frac{\bm{\rho}}{\textbf{r}}\right)
\frac{N^{\tau\left(\bm{\rho},\mathbf{r},\mathfrak{D}\setminus I\right)/2}}{\pi^{\tau\left(\bm{\rho},\mathbf{r},\mathfrak{D}\setminus I\right)}}
\\
&\times
\sum_{\bm{\rho}\in Z^{|\mathfrak{D}\setminus I|}}
\frac1{\textbf{r}}\Gamma\left(\frac{\bm{\rho}}{\textbf{r}}\right)
\frac{N^{\tau\left(\bm{\rho},\mathbf{r},\mathfrak{D}\setminus I\right)/2}}{\pi^{\tau\left(\bm{\rho},\mathbf{r},\mathfrak{D}\setminus I\right)}}
\mathfrak{B}\big(\tau\left(\mathbf{r},I\right)+\tau\left(\bm{\rho},\mathbf{r},\mathfrak{D}\setminus I\right)\big),
\end{align*}
where $\rho_{j}$ runs over the non-trivial zeros of the Riemann Zeta
function and $J_{v}\left(u\right)$ are the Bessel $J$ function of
real argument $u$ and complex order $v$ and 
\begin{align*}
\mathfrak{B}(x)
=
\mathfrak{B}_{k,h,\eta,\ell,N}(x)
=
N^{\frac{h-\eta+\ell}{4}}
\sum_{\mathbf{f}\in\left(\mathbb{N}^{+}\right)^{h-\eta}}\frac{J_{x+k+(h-\eta+\ell)/2}\left(2\pi\sqrt{N}\left\Vert \mathbf{f}\right\Vert \right)}{\left\Vert \mathbf{f}\right\Vert ^{x+k+(h-\eta+\ell)/2}}, 
\end{align*}
with
\begin{align*}
\sum_{\mathbf{f}\in\left(\mathbb{N}^{+}\right)^{c}}:=\sum_{f_{1}\geq1}\cdots\sum_{f_{c}\geq1}.
\end{align*}

The convergence of the mentioned series will be proved in the section \S \ref{conv}. The main result of this
article is the following theorem:
\begin{thm}\label{main_thm}
Let $d,h\in\mathbb{N},\,d>0,$ let $N$ be a sufficiently large integer.
Let $\mathfrak{D}:=\left\{ 1,\dots,d\right\} $ and, for every $\mathfrak{J}\subseteq\mathfrak{D}$
(or with the notation $I\subseteq\mathfrak{D})$ let $\tau\left(\mathbf{r},\mathfrak{J}\right):=\sum_{j\in\mathfrak{J}}\frac{1}{r_{j}}$, where $1\le r_1\le r_2\le \cdots\le r_d$.
Then, for $k>\frac{d+h}{2}$, we have that 
\[
\sum_{n\leq N}R_{d,h,\mathbf{r}}\left(n\right)\frac{\left(N-n\right)^{k}}{\Gamma\left(k+1\right)}=
\sum_{j=1}^{5}M_{j}\left(N,k,d,h,\mathbf{r}\right)+
O_{\mathbf{r},d,h}\left(N^{k+h/2+\tau\left(\mathbf{r},\mathfrak{D}\right)-1/r_d}\right).
\]
\end{thm}

It is important to underline that in some particular configurations of the parameters some terms of the asymptotic (but not the dominant term) could be incorporated in the error. Despite the apparently complicated form of the terms, it is not difficult to recognize
the results obtained in the previous work on this topic, for example
setting $d=2$, $h=0$ and $\mathbf{r}=\left(1,1\right)$ (the Goldbach
numbers case \cite{LanZac}) or $\mathbf{r}=\left(\ell_{1},\ell_{2}\right),\,1\leq\ell_{1}\leq\ell_{2}$
integers (the generalized Goldbach numbers case \cite{LanZac4}).
Furthermore, it is quite natural to conjecture that at least the main
term of this asymptotic is valid for $k\geq0$ instead of $k>\frac{d+h}{2}$
as suggested by similar studies but with other techniques (see, e.g.,
\cite{CanGamZac2}\cite{CanGamLangZac}). Now, let's briefly explain
the major ideas behind this theorem; one of the main tools of this
technique is the formula, due to Laplace \cite{Lap}, namely 
\begin{equation}
\frac{1}{2\pi i}\int_{\left(a\right)}v^{-s}e^{v}\dx v=\frac{1}{\Gamma\left(s\right)}\label{eq:onemaintool}
\end{equation}
 for $\mathrm{Re}\left(s\right)>0$ and $a>0$ (see formula $5.4(1)$
on page $238$ of \cite{ErdMagObeTric}), where
\[
\int_{\left(a\right)}:=\int_{a-i\infty}^{a+i\infty}.
\]

From (\ref{eq:onemaintool}) and suitable hypotheses, which we will
explain in detail in the next sections, we are able to write
\begin{equation}
\sum_{n\leq N}R_{d,h,\mathbf{r}}\left(n\right)\frac{\left(N-n\right)^{k}}{\Gamma\left(k+1\right)}=\frac{1}{2\pi i}\int_{\left(a\right)}e^{Nz}z^{-k-1}\widetilde{S}_{r_{1}}\left(z\right)\cdots\widetilde{S}_{r_{d}}\left(z\right)\omega_{2}\left(z\right)^{h}\dx z\label{eq:main}
\end{equation}
where $z=a+iy$, $a>0$, $y\in\mathbb{R}$, where
\begin{equation}
\widetilde{S}_{r}\left(z\right):=\sum_{m\geq1}\Lambda\left(m\right)e^{-m^{r}z},\qquad\omega_{2}\left(z\right):=\sum_{m\geq1}e^{-m^{2}z},\label{eq:Stildaeomega}
\end{equation}
are the series that embody the prime powers and the squares, respectively.
Since, as we will see, it is possible to develop $\widetilde{S}_{r}\left(z\right)$
as an asymptotic formula, the idea is to substitute this formula for
$\widetilde{S}_{r}\left(z\right)$, exchange the integral with all
the terms which are obtained from the various products and and finally
calculate the error. Another important aspect to emphasize is that
we work with squares, and so with $\omega_{2}\left(z\right)$, because
this function is linked to the well-known Jacobi theta $3$ function
\[
\theta_{3}\left(z\right):=\sum_{m\in\mathbb{Z}}e^{-m^{2}z}=1+2\omega_{2}\left(z\right)
\]
and $\theta_{3}\left(z\right)$ satisfies the functional equation
\[
\theta_{3}\left(z\right)=\left(\frac{\pi}{z}\right)^{1/2}\theta_{3}\left(\frac{\pi^{2}}{z}\right),\,\mathrm{Re}\left(z\right)>0
\]
(see, for example, Proposition VI.$4.3$, page $340$, of \cite{FreBus})
which implies a functional equation for $\omega_{2}\left(z\right)$
\begin{equation}
\omega_{2}\left(z\right)=\frac{1}{2}\left(\frac{\pi}{z}\right)^{1/2}-\frac{1}{2}+\left(\frac{\pi}{z}\right)^{1/2}\omega_{2}\left(\frac{\pi^{2}}{z}\right).\label{eq:funceqomega}
\end{equation}
This is fundamental for the present technique, because this functional equation
allows us to find the terms involving the Bessel $J$ function and,
since we do not have a functional equation of this type for other
powers than squares, we can only deal with this particular case.

\section{Settings}

For our purposes, we need a general version of the formula (\ref{eq:onemaintool}),
so we recall the following relations:
\begin{equation}
\frac{1}{2\pi}\int_{\mathbb{R}}\frac{e^{iDu}}{\left(a+iu\right)^{s}}\dx u=\begin{cases}
\frac{D^{s-1}e^{-aD}}{\Gamma\left(s\right)}, & D>0\\
0, & D<0
\end{cases}\label{eq:gen1}
\end{equation}
with $\mathrm{Re}\left(s\right)>0$, $\mathrm{Re}\left(a\right)>0$
and 
\begin{equation}
\frac{1}{2\pi}\int_{\mathbb{R}}\frac{1}{\left(a+iu\right)^{s}}\dx u=\begin{cases}
0, & \mathrm{Re}\left(s\right)>1\\
1/2, & \mathrm{Re}\left(s\right)=1
\end{cases}\label{eq:gen2}
\end{equation}
with $\mathrm{Re}\left(a\right)>0$ (see formulas $\left(8\right)$
and $\left(9\right)$ of \cite{Aze}). We also need an integral representation
of the Bessel $J$ function with real argument $u$ and complex order
$v$:
\begin{equation}
J_{v}\left(u\right):=\frac{\left(u/2\right)}{2\pi i}\int_{\left(a\right)}s^{-v-1}e^{s}e^{-u^{2}/\left(4s\right)}\mathrm{d} s\label{eq:bessel}
\end{equation}
for $a>0$, $u,v\in\mathbb{C}$ with $\mathrm{Re}\left(v\right)>-1$
(see, e.g., equation $\left(8\right)$ on page $177$ of \cite{Wat}).

Assume that $k>0$. From the definition of $\widetilde{S}_{r}\left(z\right)$
and $\omega_{2}\left(z\right)$ (\ref{eq:Stildaeomega}), it is not
difficult to note that 

\[
\widetilde{S}_{r_{1}}\left(z\right)\cdots\widetilde{S}_{r_{d}}\left(z\right)\omega_{2}\left(z\right)^{h}=\sum_{n\geq1}R_{d,h,\mathbf{r}}\left(n\right)e^{-nz}.
\]
 Furthermore, from (\ref{eq:gen1}) and (\ref{eq:gen2}), we have
that 
\begin{equation}
\sum_{n\leq N}R_{d,h,\mathbf{r}}\left(n\right)\frac{\left(N-n\right)^{k}}{\Gamma\left(k+1\right)}=\sum_{n\geq1}R_{d,h,\mathbf{r}}\left(n\right)\left(\frac{1}{2\pi i}\int_{\left(a\right)}e^{\left(N-n\right)z}z^{-k-1}\dx z\right).\label{eq:applicazgen}
\end{equation}
Now we want to show that it is possible to exchange the integral with
the series in the right side of (\ref{eq:applicazgen}).

By the Prime Number Theorem, we have that 
\begin{equation}
\widetilde{S}_{r_{j}}\left(a\right)\sim\frac{\Gamma\left(\frac{1}{r_{j}}\right)}{r_{j}a^{1/r_{j}}}\label{eq:PNT}
\end{equation}
as $a\rightarrow0^{+}$(see \cite{LanZac5}) and
\begin{equation}
\left|\omega_{2}\left(z\right)\right|\leq\omega_{2}\left(a\right)\leq\int_{0}^{+\infty}e^{-au^{2}}\dx u\leq a^{-1/2}\int_{0}^{+\infty}e^{-v^{2}}\dx v\ll a^{-1/2}\label{eq:omegaest}
\end{equation}
and so
\[
\sum_{n\geq1}\left|R_{d,h,\mathbf{r}}\left(n\right)e^{-nz}\right|=\sum_{n\geq1}R_{d,h,\mathbf{r}}\left(n\right)e^{-na}=\widetilde{S}_{r_{1}}\left(a\right)\cdots\widetilde{S}_{r_{d}}\left(a\right)\omega_{2}\left(a\right)^{h}
\]
\[
\ll_{\mathbf{r},d,h}a^{-\tau\left(\mathbf{r},\mathfrak{D}\right)-h/2}.
\]
From the trivial estimate
\begin{equation}
\left|e^{Nz}\right|\left|z^{-k-1}\right|\asymp e^{Na}\begin{cases}
a^{-k-1}, & \left|y\right|\leq a\\
\left|y\right|^{-k-1}, & \left|y\right|>a
\end{cases}\label{eq:trivial estimate}
\end{equation}
where $f\asymp g$ means $g\ll f\text{\ensuremath{\ll g}, we have}$
\begin{align*}
\frac{1}{2\pi i}\int_{\left(a\right)}e^{Nz}z^{-k-1}\widetilde{S}_{r_{1}}\left(z\right)\cdots\widetilde{S}_{r_{d}}\left(z\right)\omega_{2}\left(z\right)^{h}\dx z & \ll_{\mathbf{r},d,h}e^{Na}a^{-\tau\left(\mathbf{r},\mathfrak{D}\right)-h/2}\left(\int_{-a}^{a}a^{-k-1}\dx y+\int_{a}^{+\infty}y^{-k-1}\dx y\right)\\
 & \ll_{\mathbf{r},d,h}e^{Na}a^{-\tau\left(\mathbf{r},\mathfrak{D}\right)-h/2-k}
\end{align*}
for $k>0$. Then, we can exchange the integral with the series and
so we obtain the main formula \eqref{eq:main}.

\section{Lemmas}\label{conv}

In this section we present some technical lemmas that will be useful
later and some basic facts in complex analysis. First, we recall that
if $z=a+iy,\,a>0$ and $w\in\mathbb{C}$, we have that
\[
z^{-w}=\left|z\right|^{-\mathrm{Re}\left(w\right)-i\mathrm{Im}\left(w\right)}\exp\left(\left(-i\mathrm{Re}\left(w\right)+\mathrm{Im}\left(w\right)\right)\arctan\left(\frac{y}{a}\right)\right)
\]
and so
\begin{equation}
\left|z^{-w}\right|=\left|z\right|^{-\mathrm{Re}\left(w\right)}\exp\left(\mathrm{Im}\left(w\right)\arctan\left(\frac{y}{a}\right)\right).\label{eq:complex power}
\end{equation}
We also recall the Stirling formula 
\begin{equation}
\left|\Gamma\left(x+iy\right)\right|\sim\sqrt{2\pi}e^{-\pi\left|y\right|/2}\left|y\right|^{x-1/2}\label{eq:Stirling}
\end{equation}
which holds uniformly for $x\in\left[x_{1},x_{2}\right]$, $x_{1},\,x_{2}$
fixed and $\left|y\right|\rightarrow+\infty$ (see, e.g., \cite{Tit},
section $4.42$).

Now we introduce the ``explicit formula'' of $\widetilde{S}_{r}\left(z\right),\,r\in\mathbb{N}^{+}$. 
\begin{lem}
(Lemma $1$ of \cite{LanZac4}) Let $r\geq1$ be an integer, let $z=a+iy,\,a>0,\,y\in\mathbb{R}$.
Let 
\begin{equation}
T\left(z,r\right):=\frac{\Gamma\left(\frac{1}{r}\right)}{rz^{1/r}}-\frac{1}{r}\sum_{\rho}z^{-\rho/r}\Gamma\left(\frac{\rho}{r}\right).\label{eq:asmipt2}
\end{equation}
Then
\begin{equation}
\widetilde{S}_{r}\left(z\right)=T\left(z,r\right)+E\left(a,y,r\right).\label{eq:asimpt}
\end{equation}
where
\begin{equation}
\left|E\left(a,y,r\right)\right|\ll_{r}1+\left|z\right|^{1/2}\begin{cases}
1, & \left|y\right|\leq a\\
1+\log^{2}\left(\frac{\left|y\right|}{a}\right), & \left|y\right|>a.
\end{cases}\label{eq:error estimate}
\end{equation}
\end{lem}

Note that in Lemma $1$ of \cite{LanZac4} $T\left(z,r\right)$ is
defined as 
\[
T\left(z,r\right):=\frac{\Gamma\left(\frac{1}{r}\right)}{rz^{1/r}}-\frac{1}{r}\sum_{\rho}z^{-\rho/r}\Gamma\left(\frac{\rho}{r}\right)-\log\left(2\pi\right)
\]
but in our context, to make the main term combinatorically more tractable,
it is better to insert $\log\left(2\pi\right)$ in the error term
$E\left(a,y,r\right)$. Furthermore, from (\ref{eq:PNT}) and (\ref{eq:error estimate})
we immediately get the important estimate

\begin{equation}
\left|\sum_{\rho}z^{-\rho/r}\Gamma\left(\frac{\rho}{r}\right)\right|\ll_{r}a^{-1/r}+1+\left|z\right|^{1/2}\begin{cases}
1, & \left|y\right|\leq a\\
1+\log^{2}\left(\frac{\left|y\right|}{a}\right), & \left|y\right|>a
\end{cases}\label{eq:estserieszeros0}
\end{equation}
which can be rewritten, if $0<a<1$ and $r\ge1$, in the more compact
form
\begin{equation}
\left|\sum_{\rho}z^{-\rho/r}\Gamma\left(\frac{\rho}{r}\right)\right|\ll_{r}\begin{cases}
a^{-1/r}, & \left|y\right|\leq a\\
a^{-1/r}+\left|z\right|^{1/2}\log^{2}\left(\frac{2\left|y\right|}{a}\right), & \left|y\right|>a.
\end{cases}\label{eq:estserieszeros}
\end{equation}

\begin{lem}
\label{lem:generalized1}Let $\lambda\in\mathbb{N}^{+}$, $r_{1},\dots,r_{\lambda}\in\mathbb{N}^{+}$
and $\mathbf{r}:=\left(r_{1},\dots,r_{\lambda}\right)\in\left(\mathbb{N}^{+}\right)^{\lambda}$.
Let $\rho_{j}=\beta_{j}+i\gamma_{j},\,j\in\left\{ 1,\dots,\lambda\right\} $,
run over the non trivial zeros of Riemann Zeta function and $\alpha>1$
be a parameter. Then, for any fixed $b>1$ and $c\geq0$, the series
\[
\sum_{\rho_{1}:\,\gamma_{1}>0}\left(\frac{\gamma_{1}}{r_{1}}\right)^{\beta_{1}/r_{1}-1/2}\cdots\sum_{\rho_{\lambda}:\,\gamma_{\lambda}>0}\left(\frac{\gamma_{\lambda}}{r_{\lambda}}\right)^{\beta_{\lambda}/r_{\lambda}-1/2}\int_{1}^{+\infty}\log^{c}\left(bu\right)\exp\left(-\arctan\left(\frac{1}{u}\right)\tau\left(\bm{\gamma},\mathbf{r},\mathfrak{J}_{\lambda}\right)\right)\frac{\dx u}{u^{\alpha+\tau\left(\bm{\beta},\mathbf{r},\mathfrak{J}_{\lambda}\right)}}
\]
converges if $\alpha>\frac{\lambda}{2}+1$.
\end{lem}

\begin{proof}
Following the proof of Lemma $2$ of \cite{LanZac4}, we can see that
\[
\int_{1}^{+\infty}\exp\left(-\arctan\left(\frac{1}{u}\right)\tau\left(\bm{\gamma},\mathbf{r},\mathfrak{J}_{\lambda}\right)\right)\frac{\dx u}{u^{\alpha+\tau\left(\bm{\beta},\mathbf{r},\mathfrak{J}_{\lambda}\right)}}
\ll_{\alpha,\mathbf{r}}\tau\left(\bm{\gamma},\mathbf{r},\mathfrak{J}_{\lambda}\right)^{1-\alpha-\tau\left(\bm{\beta},\mathbf{r},\mathfrak{J}_{\lambda}\right)}\int_{0}^{+\infty}e^{-w}w^{\alpha+\tau\left(\bm{\beta},\mathbf{r},\mathfrak{J}_{\lambda}\right)-2}\dx w
\]
and the integral converges since $0<\beta_{j}<1,\,j=1,\dots,\lambda$
and $\alpha>1.$ Hence, we have to consider
\[
\sum_{\rho_{1}:\,\gamma_{1}>0}\cdots\sum_{\rho_{\lambda}:\,\gamma_{\lambda}>0}\frac{\left(\frac{\gamma_{1}}{r_{1}}\right)^{\beta_{1}/r_{1}-1/2}\cdots\left(\frac{\gamma_{\lambda}}{r_{\lambda}}\right)^{\beta_{\lambda}/r_{\lambda}-1/2}}{\tau\left(\bm{\gamma},\mathbf{r},\mathfrak{J}_{\lambda}\right)^{\alpha+\tau\left(\bm{\beta},\mathbf{r},\mathfrak{J}_{\lambda}\right)-1}}.
\]
Now, it is not difficult to note that
\begin{equation}
\frac{\left(\frac{\gamma_{1}}{r_{1}}\right)^{\beta_{1}/r_{1}}\cdots\left(\frac{\gamma_{\lambda}}{r_{\lambda}}\right)^{\beta_{\lambda}/r_{\lambda}}}
{\tau\left(\bm{\gamma},\mathbf{r},\mathfrak{J}_{\lambda}\right)^{\tau\left(\bm{\beta},\mathbf{r},\mathfrak{J}_{\lambda}\right)}}\leq1\label{eq:stima semplice}
\end{equation}
so we analyze
\[
\sum_{\rho_{1}:\,\gamma_{1}>0}\cdots\sum_{\rho_{\lambda}:\,\gamma_{\lambda}>0}\frac{\left(\frac{\gamma_{1}}{r_{1}}\right)^{-1/2}\cdots\left(\frac{\gamma_{\lambda}}{r_{\lambda}}\right)^{-1/2}}
{\tau\left(\bm{\gamma},\mathbf{r},\mathfrak{J}_{\lambda}\right)^{\alpha-1}}
\le \sum_{\rho_{1}:\,\gamma_{1}>0}\left(\frac{\gamma_{1}}{r_{1}}\right)^{-\frac12-\frac{\alpha-1}{\lambda}}\cdots\sum_{\rho_{\lambda}:\,\gamma_{\lambda}>0}\left(\frac{\gamma_{\lambda}}{r_{\lambda}}\right)^{-\frac12-\frac{\alpha-1}{\lambda}}
\]
by the inequality of arithmetic and geometric means.
From the asymptotic formula of $N\left(T\right)$, where $N\left(T\right)$
is the number of non-trivial zeros of the Riemann zeta function with
imaginary part $0\leq\gamma\leq T$, it is not difficult to prove,
putting $\gamma\left(k\right)$ the imaginary part of the
$k$-th non-trivial zeros of $\zeta\left(s\right)$, that
\[
\gamma\left(k\right)\sim\frac{2\pi k}{\log\left(k\right)}
\]
as $k\rightarrow+\infty$.  So the series converges if $\alpha>\frac{\lambda}{2}+1.$
The treatment is similar for the case $c>0$.
\end{proof}
\begin{lem}
\label{lem:generalizlemnew}Let $N,\lambda,\alpha$ be positive integers,
let $h\in\mathbb{Q}^{+},$ let $\rho_{j}=\beta_{j}+i\gamma_{j},\,j\in\left\{ 1,\dots,\lambda\right\} $,
run over the non-trivial zeros of the Riemann Zeta function, $\left\Vert \cdot\right\Vert $
the Euclidean norm in $\mathbb{R}^{d},\,d\in\mathbb{N}^{+}$ and $k>0$
a real number. For sake of simplicity we define $\delta:=\sum_{j=1}^{\lambda}\gamma_{j}$.
Then, for every fixed integer $b>1$ and $c>0$, 
\[
\sum_{\rho_{1}:\,\gamma_{1}>0}\cdots\sum_{\rho_{\lambda}:\,\gamma_{\lambda}>0}\frac{\gamma_{1}^{-\frac{1}{2}}\cdots\gamma_{\lambda}^{-\frac{1}{2}}}{\delta^{k+h+\alpha}}
\sum_{\mathbf{f}\in\left(\mathbb{N}^{+}\right)^{\alpha}}
\int_{0}^{\delta}v^{k-1+h+\alpha+\tau\left(\bm{\beta},\mathbf{r},\mathfrak{J}_{\lambda}\right)}
 e^{-\left\Vert \mathbf{f}\right\Vert ^{2}Nv^{2}/\delta^{2}-v}
 \log^{2c}\left(\frac{b\,\delta}{v}\right)\dx v
\]
converges if $k>\frac{\lambda}{2}-h.$
\end{lem}

\begin{proof}
We consider the integral
\begin{equation}
\sum_{\rho_{1}:\,\gamma_{1}>0}\cdots\sum_{\rho_{\lambda}:\,\gamma_{\lambda}>0}\frac{\gamma_{1}^{-\frac{1}{2}}\cdots\gamma_{\lambda}^{-\frac{1}{2}}}{{\delta}^{k+h+\alpha}}
\sum_{\mathbf{f}\in\left(\mathbb{N}^{+}\right)^{\alpha}}
\int_{0}^{\delta}v^{k-1+h+\alpha+\tau\left(\bm{\beta},\mathbf{r},\mathfrak{J}_{\lambda}\right)}e^{-\left\Vert \mathbf{f}\right\Vert ^{2}Nv^{2}{\delta}^{-2}}\exp\left(-v\right)\dx v.\label{int lemma}
\end{equation}
Now we claim that we can exchange the integral with the multiple series
$\sum_{\mathbf{f}\in\left(\mathbb{N}^{+}\right)^{\alpha}}$. To show
this we consider 
\[
\int_{0}^{\delta} v^{k-1+h+\alpha+\tau\left(\bm{\beta},\mathbf{r},\mathfrak{J}_{\lambda}\right)}
\sum_{f_{1}\geq1}e^{-f_{1}Nv^{2}{\delta}^{-2}}\omega_{2}^{\alpha-1}(Nv^{2}\delta^{-2})\exp\left(-v\right)\dx v.
\]
Now, since for every $M\geq1$ we have
\[
\sum_{f_{1}\leq M}e^{-f_{1}Nv^{2}\delta^{-2}}\leq\sum_{f_{1}\geq1}e^{-f_{1}Nv^{2}\delta^{-2}}
=
\omega_{2}(Nv^{2}\delta^{-2})
\ll_{N}
\frac{\delta}{v}
\]
from (\ref{eq:omegaest}) and so we have to deal with
\[
\int_{0}^{\delta}v^{k-2+h+\alpha+\tau\left(\bm{\beta},\mathbf{r},\mathfrak{J}_{\lambda}\right)}\omega_{2}^{\alpha-1}(Nv^{2}\delta^{-2})\exp\left(-v\right)\dx v\ll_{N,\alpha}
\delta^{\alpha-1}\int_{0}^{\delta}v^{k+h-1+\tau\left(\bm{\beta},\mathbf{r},\mathfrak{J}_{\lambda}\right)}\exp\left(-v\right)\dx v
\]
which is convergent since $k>0$, then we obtain
\[
\sum_{f_{1}\geq1}\int_{0}^{\delta}v^{k-1+h+\alpha+\frac{\beta_{1}}{r_{1}}+\dots+\frac{\beta_{\lambda}}{r_{\lambda}}}
e^{-f_{1}Nv^{2}{\delta}^{-2}}
\omega_{2}^{\alpha-1}(Nv^{2}\delta^{-2})\exp\left(-v\right)\dx v
\]
by the Dominated Convergence Theorem. Clearly, we can repeat the same
argument for every factor in the product $\omega_{2}^{\alpha-1}(Nv^{2}\delta^{-2})$
and so we can write (\ref{int lemma}) as
\[
\sum_{\rho_{1}:\,\gamma_{1}>0}\cdots\sum_{\rho_{\lambda}:\,\gamma_{\lambda}>0}\frac{\gamma_{1}^{-\frac{1}{2}}\cdots\gamma_{\lambda}^{-\frac{1}{2}}}{\delta^{k+h+\alpha}}
\int_{0}^{\delta}v^{k-1+h+\alpha+\tau\left(\bm{\beta},\mathbf{r},\mathfrak{J}_{\lambda}\right)}
\omega_{2}^{\alpha}(Nv^{2}\delta^{-2})\exp(-v)\dx v
\]
and again using  (\ref{eq:omegaest}) 
we have to deal with
\[
\sum_{\rho_{1}:\,\gamma_{1}>0}\cdots\sum_{\rho_{\lambda}:\,\gamma_{\lambda}>0}\frac{\gamma_{1}^{-\frac{1}{2}}\cdots\gamma_{\lambda}^{-\frac{1}{2}}}{{\delta}^{k+h+\alpha}}
\int_{0}^{\delta}v^{k-1+h+\tau\left(\bm{\beta},\mathbf{r},\mathfrak{J}_{\lambda}\right)}\exp\left(-v\right)\dx v.
\]
Now, since $k+h+\tau\left(\bm{\beta},\mathbf{r},\mathfrak{J}_{\lambda}\right)>0$, then 
\[
\int_{0}^{\delta}v^{k-1+h+\tau\left(\bm{\beta},\mathbf{r},\mathfrak{J}_{\lambda}\right)}\exp\left(-v\right)\dx v\ll\int_{0}^{+\infty}v^{k-1+h+\tau\left(\bm{\beta},\mathbf{r},\mathfrak{J}_{\lambda}\right)}\exp\left(-v\right)\dx v<+\infty.
\]
Then, from arithmetic mean - geometric mean inequality, we get
\begin{equation}
\sum_{\rho_{1}:\,\gamma_{1}>0}\cdots\sum_{\rho_{\lambda}:\,\gamma_{\lambda}>0}\frac{\gamma_{1}^{-\frac{1}{2}}\cdots\gamma_{\lambda}^{-\frac{1}{2}}}{\delta^{k+h+\alpha}}
\ll
\sum_{\rho_{1}:\,\gamma_{1}>0}\gamma_{1}^{-\frac{k}{\lambda}-\frac{1}{2}-\frac{h}{\lambda}}\cdots\sum_{\rho_{\lambda}:\,\gamma_{\lambda}>0}\gamma_{\lambda}^{-\frac{k}{\lambda}-\frac{1}{2}-\frac{h}{\lambda}}\label{eq:sumgammageneralized}
\end{equation}
and the series converges if $k>\frac{\lambda}{2}-h.$ Clearly, if
we have a log factor into the integral the bound for $k$ is the same.
Indeed, we note that
\[
\log^{2c}\left(b\frac{\delta}{v}\right)
\ll
\log^{2c}(\delta)+\log^{2c}\left(bv\right)
\]
and
\[
\log^{2c}(\delta)\leq\log^{2c}\left(\lambda\max_{\gamma_{j},\,j=1,\dots,\lambda}\gamma_{j}\right):=\log^{2c}\left(\lambda\gamma_{\star}\right)
\]

so we have in (\ref{eq:sumgammageneralized}) one series such that
\[
\sum_{\rho_{\star}:\,\gamma_{\star}>0}\gamma_{\star}^{-\frac{k}{\lambda}-\frac{1}{2}-\frac{h}{\lambda}}\log^{2c}\left(\lambda\gamma_{\star}\right)
\]
and clearly the log factor does not affect the bound for $k$; if
we have
\[
\int_{0}^{+\infty}v^{k-1+h+\tau\left(\bm{\beta},\mathbf{r},\mathfrak{J}_{\lambda}\right)}\exp\left(-v\right)\log^{2c}\left(bv\right)\dx v
\]
again, we have the same bounds for $k$ and so the Lemma is proved. 
\end{proof}

\section{proof of the main theorem}

In this section we prove the main theorem. We first show that the
error bound in the main formula is ``small'', then we prove that
all the exchange of symbols is justified and finally we evaluate
the integrals.

\subsection{Error term}

From (\ref{eq:asmipt2}), (\ref{eq:asimpt}) and following the subdivision
in \cite{CanGamZac}, formula $(2)$, we can write
\[
\widetilde{S}_{r_{1}}\left(z\right)\cdots\widetilde{S}_{r_{d}}\left(z\right)=T\left(z,r_{1}\right)\cdots T\left(z,r_{d}\right)
\]
\[
+\sum_{j=1}^{d}E\left(a,y,r_{j}\right)\left(\prod_{i\neq j}\widetilde{S}_{r_{i}}\left(z\right)\right)+\sum_{\underset{{\scriptstyle \left|I\right|\geq2}}{I\subseteq\mathfrak{D}}}c_{d}\left(I\right)\left(\prod_{i\in\mathfrak{D}\setminus I}T\left(z,r_{i}\right)\right)\left(\prod_{\ell\in I}E\left(a,y,r_{\ell}\right)\right)
\]
for some suitable coefficients $c_{d}\left(I\right)$, so we get
\begin{align*}
\frac{1}{2\pi i}\int_{\left(a\right)}e^{Nz}z^{-k-1}&\widetilde{S}_{r_{1}}\left(z\right)\cdots\widetilde{S}_{r_{d}}\left(z\right)\omega_{2}\left(z\right)^{h}\dx z  =\frac{1}{2\pi i}\int_{\left(a\right)}e^{Nz}z^{-k-1}T\left(z,r_{1}\right)\cdots T\left(z,r_{d}\right)\omega_{2}\left(z\right)^{h}\dx z\\
 & +\frac{1}{2\pi i}\sum_{j=1}^{d}\int_{\left(a\right)}e^{Nz}z^{-k-1}E\left(a,y,r_{j}\right)\left(\prod_{i\neq j}\widetilde{S}_{r_{i}}\left(z\right)\right)\omega_{2}\left(z\right)^{h}\dx z\\
 & +\sum_{\underset{{\scriptstyle \left|I\right|\geq2}}{I\subseteq\mathfrak{D}}}c_{d}\left(I\right)\int_{\left(a\right)}e^{Nz}z^{-k-1}\left(\prod_{i\in\mathfrak{D}\setminus I}T\left(z,r_{i}\right)\right)\left(\prod_{\ell\in I}E\left(a,y,r_{\ell}\right)\right)\omega_{2}\left(z\right)^{h}\dx z\\
 & =:A_{1}+A_{2}+A_{3}.
\end{align*}
 Now we have to estimate the error term. From (\ref{eq:PNT}), (\ref{eq:omegaest})
and (\ref{eq:error estimate}) we obtain
\[
\left|A_{2}\right|\ll\sum_{j=1}^{d}\int_{\left(a\right)}\left|e^{Nz}\right|\left|z^{-k-1}\right|\left|E\left(a,y,r_{j}\right)\right|\prod_{i\neq j}\left|\widetilde{S}_{r_{i}}\left(z\right)\right|\left|\omega_{2}\left(z\right)^{h}\right|\dx y
\]
\[
\ll_{\mathbf{r},d,h}e^{Na}a^{-h/2}\sum_{j=1}^{d}a^{-\tau\left(\mathbf{r},\mathfrak{D}\right)+\frac{1}{r_{j}}}\left(\int_{0}^{a}a^{-k-1}\left(1+a^{1/2}\right)\dx y+\int_{a}^{+\infty}y^{-k-1}\left(1+y^{1/2}\left(1+\log^{2}\left(\frac{y}{a}\right)\right)\right)\dx y\right)
\]
\begin{equation}
\ll_{\mathbf{r},d,h}e^{Na}a^{-k-h/2}\sum_{j=1}^{d}a^{-\tau\left(\mathbf{r},\mathfrak{D}\right)+\frac{1}{r_{j}}}\label{eq:main error term}
\end{equation}
for $k>0$. 

For the estimation of $A_{3}$ we fix $I\subseteq\mathfrak{D}$
and we consider
\[
\left|A_{3.I}\right|:=\int_{\left(a\right)}\left|e^{Nz}\right|\left|z^{-k-1}\right|\prod_{i\in\mathfrak{D}\setminus I}\left|T\left(z,r_{i}\right)\right|\prod_{\ell\in I}\left|E\left(a,y,r_{\ell}\right)\right|\left|\omega_{2}\left(z\right)^{h}\right|\dx y.
\]
We know from (\ref{eq:PNT}) and (\ref{eq:asimpt}) that 
\[
\left|T\left(z,r\right)\right|\ll_{r}a^{-1/r}+\left|E\left(a,y,r\right)\right|
\]
hence, using formula \eqref{eq:omegaest} it is enough to work with
\[
a^{-h/2}\int_{\left(a\right)}\left|e^{Nz}\right|\left|z^{-k-1}\right|\prod_{i\in\mathfrak{D}\setminus I}\left(a^{-1/r_{i}}+\left|E\left(a,y,r_{i}\right)\right|\right)\prod_{\ell\in I}\left|E\left(a,y,r_{\ell}\right)\right|\dx y.
\]
Now, observing that
\[
\prod_{i\in\mathfrak{D}\setminus I}\left(a^{-1/r_{i}}+\left|E\left(a,y,r_{i}\right)\right|\right)=\sum_{\mathfrak{J}\subseteq\mathfrak{D}\setminus I}a^{-\tau\left(\mathbf{r},\mathfrak{J}\right)}\prod_{i\in\mathfrak{D}\setminus\left(I\cup\mathfrak{J}\right)}\left|E\left(a,y,r_{i}\right)\right|
\]
we have by \eqref{eq:error estimate},
\begin{align*}
\left|A_{3.I}\right|\ll_{\mathbf{r},d,h}&e^{Na}a^{-h/2}\sum_{\mathfrak{J}\subseteq\mathfrak{D}\setminus I}\int_{\left(a\right)}a^{-\tau\left(\mathbf{r},\mathfrak{J}\right)}\prod_{i\in\mathfrak{D}\setminus\left(I\cup\mathfrak{J}\right)}\left|E\left(a,y,r_{i}\right)\right|\prod_{\ell\in I}\left|E\left(a,y,r_{\ell}\right)\right|\left|z^{-k-1}\right|\dx y
\\
\ll_{\mathbf{r},d,h}&e^{Na}a^{-h/2}\sum_{\mathfrak{J}\subseteq\mathfrak{D}\setminus I}a^{-\tau\left(\mathbf{r},\mathfrak{J}\right)}\int_{0}^{a}a^{-k-1}\left(1+a^{1/2}\right)^{\left|\mathfrak{D}\setminus\mathfrak{J}\right|}\dx y
\\
&+e^{Na}a^{-h/2}\sum_{\mathfrak{J}\subseteq\mathfrak{D}\setminus I}a^{-\tau\left(\mathbf{r},\mathfrak{J}\right)}\int_{a}^{+\infty}y^{-k-1}\left(1+y^{1/2}\left(1+\log^{2}\left(\frac{y}{a}\right)\right)\right)^{\left|\mathfrak{D}\setminus\mathfrak{J}\right|}\dx y
\\
\ll_{\mathbf{r},d,h}&e^{Na}a^{-k-h/2}\sum_{\mathfrak{J}\subseteq\mathfrak{D}\setminus I}a^{-\tau\left(\mathbf{r},\mathfrak{J}\right)}
\end{align*}
for $k>\frac{\left|\mathfrak{D}\setminus\mathfrak{J}\right|}{2}$
and since this inequality must holds for all subsets $\mathfrak{J}\subseteq\mathfrak{D},$
we have to assume $k>\frac{d}{2}$. Hence

\[
\left|A_{3}\right|\ll_{\mathbf{r},d,h}e^{Na}a^{-k-h/2}\sum_{\underset{{\scriptstyle \left|I\right|\geq2}}{I\subseteq\mathfrak{D}}}\sum_{\mathfrak{J}\subseteq\mathfrak{D}\setminus I}a^{-\tau\left(\mathbf{r},\mathfrak{J}\right)}.
\]

Now we take $a=1/N$ and we observe that
\[
\sum_{\underset{{\scriptstyle \left|I\right|\geq2}}{I\subseteq\mathfrak{D}}}\sum_{\mathfrak{J}\subseteq\mathfrak{D}\setminus I}N^{\tau\left(\mathbf{r},\mathfrak{J}\right)}\ll_{d}\max_{\underset{{\scriptstyle \left|I\right|\geq2}}{I\subseteq\mathfrak{D}}}\max_{\mathfrak{J}\subseteq\mathfrak{D}\setminus I}N^{\tau\left(\mathbf{r},\mathfrak{J}\right)}\ll_{d}N^{\tau\left(\mathbf{r},\mathfrak{D}\right)-\frac{1}{r_{j_{1}}}-\frac{1}{r_{j_{2}}}}
\ll_{d}
N^{k+h/2+\tau\left(\mathbf{r},\mathfrak{D}\right)-1/r_d}
\]
remembering that $1\le r_1\le r_2\le \ldots\le r_d$.  This error term is compatible with that of  Theorem \ref{main_thm}.

\subsection{Evaluation of the main term}

According to (14) we rewrite $A_{1}$ in the following form
\begin{align*}
A_{1} & =\frac{1}{2\pi i}\int_{\left(1/N\right)}e^{Nz}z^{-k-1}T\left(z,r_{1}\right)\cdots T\left(z,r_{d}\right)\omega_{2}\left(z\right)^{h}\dx z\\
 & =\frac{1}{2\pi i}
\frac1{\textbf{r}}\Gamma\left(\frac{1}{\textbf{r}}\right)
 \int_{\left(1/N\right)}e^{Nz}z^{-k-1-\tau\left(\mathbf{r},\mathfrak{D}\right)}\omega_{2}\left(z\right)^{h}\dx z\\
 & +\frac{\left(-1\right)^{d}}{2\pi i}\int_{\left(1/N\right)}e^{Nz}z^{-k-1}\left(
 \sum_{\bm{\rho}\in Z^{d}}\frac1{\textbf{r}}\Gamma\left(\frac{\bm{\rho}}{\textbf{r}}\right)
z^{-\tau(\bm{\rho},\textbf{r},\mathfrak{D})}
\right)\omega_{2}\left(z\right)^{h}\dx z
 \\
 & +\frac{1}{2\pi i}\sum_{\underset{{\scriptstyle \left|I\right|\geq1}}{I\subseteq\mathfrak{D}}}\left(-1\right)^{\left|\mathfrak{D}\setminus I\right|}\int_{\left(1/N\right)}e^{Nz}z^{-k-1-\tau\left(\mathbf{r},I\right)} 
 \left(
 \sum_{\bm{\rho}\in Z^{|\mathfrak{D}\setminus I|}}
\frac1{\textbf{r}}\Gamma\left(\frac{\bm{\rho}}{\textbf{r}}\right)
z^{-\tau\left(\bm{\rho},\mathbf{r},\mathfrak{D}\setminus I\right)}
\right)\omega_{2}\left(z\right)^{h}\dx z\\
 & =:I_{1}+I_{2}+I_{3}.
\end{align*}

$I_1$ corresponds to the terms $M_1$ and $M_2$ of Theorem \ref{main_thm}, $I_2$ corresponds to the terms $M_3$ and $M_4$ and finally $I_3$ corresponds to $M_3$.

\subsubsection{Evaluation of $I_{1}$}

We study $I_{1}$. By (\ref{eq:funceqomega}) and the binomial theorem,
we get
\begin{align*}
\omega_{2}\left(z\right)^{h}=&\left(\frac{1}{2}\left(\frac{\pi}{z}\right)^{1/2}-\frac{1}{2}+\left(\frac{\pi}{z}\right)^{1/2}\omega_{2}\left(\frac{\pi^{2}}{z}\right)\right)^{h}
\\
=&\sum_{\eta=0}^{h}\frac{\dbinom{h}{\eta}}{2^{\eta}}\left(\left(\frac{\pi}{z}\right)^{1/2}-1\right)^{\eta}\omega_{2}\left(\frac{\pi^{2}}{z}\right)^{h-\eta}\left(\frac{\pi}{z}\right)^{\frac{h-\eta}{2}}
\\
=&\sum_{\eta=0}^{h}\frac{\dbinom{h}{\eta}}{2^{\eta}}\sum_{\ell=0}^{\eta}\dbinom{\eta}{\ell}\left(-1\right)^{\eta-\ell}\left(\frac{\pi}{z}\right)^{\frac{h-\eta+\ell}{2}}\omega_{2}\left(\frac{\pi^{2}}{z}\right)^{h-\eta}
\end{align*}

and so
\begin{align*}
I_{1} & =\frac{1}{2\pi i}
\frac1{\textbf{r}}\Gamma\left(\frac{1}{\textbf{r}}\right)
\int_{\left(1/N\right)}e^{Nz}z^{-k-1-\tau\left(\mathbf{r},\mathfrak{D}\right)}\omega_{2}\left(z\right)^{h}\dx z\\
 & =\frac{1}{2\pi i}\sum_{\eta=0}^{h}\frac{\dbinom{h}{\eta}}{2^{\eta}}\sum_{\ell=0}^{\eta}\dbinom{\eta}{\ell}\pi^{\frac{h-\eta+\ell}{2}}\left(-1\right)^{\eta-\ell}
\frac1{\textbf{r}}\Gamma\left(\frac{1}{\textbf{r}}\right)
\int_{\left(1/N\right)}e^{Nz}z^{-k-1-\tau\left(\mathbf{r},\mathfrak{D}\right)-\frac{h-\eta+\ell}{2}}\omega_{2}\left(\frac{\pi^{2}}{z}\right)^{h-\eta}\dx z.
\end{align*}
Our main goal is to show that, for a suitable $k$, we can exchange
the integral with the involved series; in this case, with the series
related to $\omega_{2}$. We consider two cases: if $\eta=h$ we get
\[
I_{1,1}:=\frac{1}{2^{h+1}\pi i}\sum_{\ell=0}^{h}\dbinom{h}{\ell}\pi^{\frac{\ell}{2}}\left(-1\right)^{h-\ell}
\frac1{\textbf{r}}\Gamma\left(\frac{1}{\textbf{r}}\right)
\int_{\left(1/N\right)}e^{Nz}z^{-k-1-\tau\left(\mathbf{r},\mathfrak{D}\right)-\frac{\ell}{2}}\dx z
\]
which corresponds to the term $M_1$ in Theorem \ref{main_thm} and, from the substitution $Nz=u$ and (\ref{eq:onemaintool}), we
get
\begin{align*}
I_{1,1}=&\frac{1}{2^{h+1}\pi i}\sum_{\ell=0}^{h}\dbinom{h}{\ell}\pi^{\frac{\ell}{2}}\left(-1\right)^{h-\ell}N^{k+\tau\left(\mathbf{r},d\right)+\frac{\ell}{2}}
\frac1{\textbf{r}}\Gamma\left(\frac{1}{\textbf{r}}\right)
\int_{\left(1\right)}e^{u}u^{-k-1-\tau\left(\mathbf{r},\mathfrak{D}\right)-\frac{\ell}{2}}\dx u
\\
=&\frac{1}{2^{h}}\sum_{\ell=0}^{h}\dbinom{h}{\ell}\frac{\pi^{\frac{\ell}{2}}\left(-1\right)^{h-\ell}N^{k+\tau\left(\mathbf{r},\mathfrak{D}\right)+\frac{\ell}{2}}}{\Gamma\left(k+1+\tau\left(\mathbf{r},\mathfrak{D}\right)+\frac{\ell}{2}\right)}
\frac1{\textbf{r}}\Gamma\left(\frac{1}{\textbf{r}}\right)
\end{align*}
for $k+1+\tau\left(\mathbf{r},\mathfrak{D}\right)+\frac{\ell}{2}>0$,
which is trivially true if $k>0$. Now, fix $1\leq\lambda\leq h-\eta$.

We consider the general case
\begin{equation}
I_{1,2,\lambda}:=\sum_{f_{1}\geq1}\cdots\sum_{f_{\lambda}\geq1}\int_{\left(1/N\right)}\left|e^{Nz}\right|\left|z\right|^{-k-1-\tau\left(\mathbf{r},\mathfrak{D}\right)-\frac{h-\eta+\ell}{2}}e^{-\pi^{2}\mathrm{Re}\left(1/z\right)\left(f_{1}^{2}+\dots+f_{\lambda}^{2}\right)}\left|\omega_{2}\left(\frac{\pi^{2}}{z}\right)\right|^{h-\eta-\lambda}\left|\dx z\right|.\label{eq:I12 general case}
\end{equation}
Note that if $I_{1,2,\lambda}$ converges for all $\lambda,$ the
exchange between series and integral is justified. By the trivial
estimate
\[
\mathrm{Re}\left(\frac{1}{z}\right)=\frac{N}{1+y^{2}N^{2}}\gg\begin{cases}
N, & \left|y\right|\leq1/N\\
1/\left(Ny^{2}\right), & \left|y\right|>1/N,
\end{cases}
\]
by (\ref{eq:trivial estimate}) and by (\ref{eq:omegaest}), we obtain 
\begin{align*}
I_{1,2,\lambda}\ll_{h,\eta,\lambda}&\sum_{\mathbf{f}\in\left(\mathbb{N}^{+}\right)^{\lambda}}\int_{0}^{1/N}N^{k+1+\tau\left(\mathbf{r},\mathfrak{D}\right)+\frac{\lambda+\ell}{2}}e^{-\pi^{2}N\left(\left\Vert \mathbf{f}\right\Vert ^{2}\right)}\dx y
\\
&+N^{h-\eta-\lambda}\sum_{\mathbf{f}\in\left(\mathbb{N}^{+}\right)^{\lambda}}\int_{1/N}^{+\infty}y^{-k-1-\tau\left(\mathbf{r},\mathfrak{D}\right)+\frac{h-\eta-\ell}{2}-\lambda}e^{-\frac{\pi^{2}\left(\left\Vert \mathbf{f}\right\Vert ^{2}\right)}{Ny^{2}}}\dx y.
\end{align*}
The first integral and the series trivially converge since $N$ is
positive, then we can consider only the second integral. Making the
substitution $v=\frac{\pi^{2}\left(\left\Vert \mathbf{f}\right\Vert ^{2}\right)}{Ny^{2}},$ we
get 
\begin{align*}
\sum_{\mathbf{f}\in\left(\mathbb{N}^{+}\right)^{\lambda}}&\int_{1/N}^{+\infty}y^{-k-1-\tau\left(\mathbf{r},\mathfrak{D}\right)+\frac{h-\eta-\ell}{2}-\lambda}e^{-\frac{\pi^{2}\left(\left\Vert \mathbf{f}\right\Vert ^{2}\right)}{Ny^{2}}}\dx y
\\
&\ll_{N,h,\eta,\lambda}
\sum_{\mathbf{f}\in\left(\mathbb{N}^{+}\right)^{\lambda}}
\left\Vert \mathbf{f}\right\Vert ^{-\big(k+\tau\left(\mathbf{r},\mathfrak{D}\right)-\frac{h-\eta-\ell}{2}+\lambda\big)}
\int_{0}^{+\infty}v^{\frac12 \big(k+\tau\left(\mathbf{r},\mathfrak{D}\right)-\frac{h-\eta-\ell}{2}+\lambda-1\big)-1}e^{-v}\dx v.
\end{align*}

Now, the integral is convergent if $k+\tau\left(\mathbf{r},\mathfrak{D}\right)-\frac{h-\eta-\ell}{2}+\lambda-1>0$,
which means $k>-\tau\left(\mathbf{r},\mathfrak{D}\right)+\frac{h-\eta-\ell}{2}-\lambda+1$
and the series is convergent if $k+\tau\left(\mathbf{r},\mathfrak{D}\right)-\frac{h-\eta-\ell}{2}+\lambda>\lambda$,
from the inequality of arithmetic and geometric means, and so $k>-\tau\left(\mathbf{r},\mathfrak{D}\right)+\frac{h-\eta-\ell}{2}$.
Since the inequalities must holds for all $1\leq\lambda\leq h-\eta$,
for all $0\leq\ell\leq\eta$ and for all $0\leq\eta\leq h-1$, we
can conclude that we can exchange all the series with the integral
if $k>-\tau\left(\mathbf{r},\mathfrak{D}\right)+\frac{h}{2}.$ Hence,
using (\ref{eq:bessel}) we can finally write
\begin{align*}
I_{1,2} & =\frac{1}{2\pi i}\sum_{\eta=0}^{h-1}\frac{\dbinom{h}{\eta}}{2^{\eta}}\sum_{\ell=0}^{\eta}\dbinom{\eta}{\ell}\pi^{\frac{h-\eta+\ell}{2}}\left(-1\right)^{\eta-\ell}
\frac1{\textbf{r}}\Gamma\left(\frac{1}{\textbf{r}}\right)
\sum_{\mathbf{f}\in\left(\mathbb{N}^{+}\right)^{h-\eta}}\int_{\left(1/N\right)}e^{Nz}z^{-k-1-\tau\left(\mathbf{r},\mathfrak{D}\right)-\frac{h-\eta+\ell}{2}}e^{-\frac{\pi^{2}\left\Vert \mathbf{f}\right\Vert ^{2}}{z}}\dx z\\
 & =\frac{N^{k+\tau\left(\mathbf{r},\mathfrak{D}\right)}}{2\pi i}\sum_{\eta=0}^{h-1}\frac{\dbinom{h}{\eta}}{2^{\eta}}\sum_{\ell=0}^{\eta}\dbinom{\eta}{\ell}\left(N\pi\right)^{\frac{h-\eta+\ell}{2}}\left(-1\right)^{\eta-\ell}
\frac1{\textbf{r}}\Gamma\left(\frac{1}{\textbf{r}}\right)
\sum_{\mathbf{f}\in\left(\mathbb{N}^{+}\right)^{h-\eta}}\int_{\left(1\right)}e^{u}u^{-k-1-\tau\left(\mathbf{r},\mathfrak{D}\right)-\frac{h-\eta+\ell}{2}}e^{-\frac{\pi^{2}\left\Vert \mathbf{f}\right\Vert ^{2}N}{u}}\dx u\\
 & =\frac{N^{\frac{k+\tau\left(\mathbf{r},\mathfrak{D}\right)}{2}}}{\pi^{k+\tau\left(\mathbf{r},\mathfrak{D}\right)}}\sum_{\eta=0}^{h-1}\frac{\dbinom{h}{\eta}}{2^{\eta}}\sum_{\ell=0}^{\eta}\dbinom{\eta}{\ell}N^{\frac{h-\eta+\ell}{4}}\left(-1\right)^{\eta-\ell}
\frac1{\textbf{r}}\Gamma\left(\frac{1}{\textbf{r}}\right)
\sum_{\mathbf{f}\in\left(\mathbb{N}^{+}\right)^{h-\eta}}\frac{J_{k+\tau\left(\mathbf{r},\mathfrak{D}\right)+\frac{h-\eta+\ell}{2}}\left(2\pi\sqrt{N}\left\Vert \mathbf{f}\right\Vert \right)}{\left\Vert \mathbf{f}\right\Vert ^{k+\tau\left(\mathbf{r},\mathfrak{D}\right)+\frac{h-\eta+\ell}{2}}}
\end{align*}
for $k>-\tau\left(\mathbf{r},\mathfrak{D}\right)+\frac{h}{2}.$ This term corresponds to $M_2$ in Theorem \ref{main_thm}.

\subsubsection{Evaluation of $I_{2}$}

As in the previous case, we split the integral into two pieces
\begin{align*}
I_{2} & =\frac{\left(-1\right)^{d}}{2\pi i}\int_{\left(1/N\right)}e^{Nz}z^{-k-1} \left( 
 \sum_{\bm{\rho}\in Z^{d}}\frac1{\textbf{r}}\Gamma\left(\frac{\bm{\rho}}{\textbf{r}}\right)
z^{-\tau(\bm{\rho},\textbf{r},\mathfrak{D})}
\right)\omega_{2}\left(z\right)^{h}\dx z
\\
 & =\frac{\left(-1\right)^{d}}{2\pi i}\sum_{\eta=0}^{h}\frac{\dbinom{h}{\eta}}{2^{\eta}}\sum_{\ell=0}^{\eta}\dbinom{\eta}{\ell}\pi^{\frac{h-\eta+\ell}{2}}\left(-1\right)^{\eta-\ell}\int_{\left(1/N\right)}e^{Nz}z^{-k-1-\frac{h-\eta+\ell}{2}}\left( 
 \sum_{\bm{\rho}\in Z^{d}}\frac1{\textbf{r}}\Gamma\left(\frac{\bm{\rho}}{\textbf{r}}\right)
z^{-\tau(\bm{\rho},\textbf{r},\mathfrak{D})}
 \right)\omega_{2}\left(\frac{\pi^{2}}{z}\right)^{h-\eta}\dx z
 \\
 & =\frac{\left(-1\right)^{d}}{2^{h+1}\pi i}\sum_{\ell=0}^{h}\dbinom{h}{\ell}\pi^{\frac{\ell}{2}}\left(-1\right)^{h-\ell}\int_{\left(1/N\right)}e^{Nz}z^{-k-1-\frac{\ell}{2}}\left( 
 \sum_{\bm{\rho}\in Z^{d}}\frac1{\textbf{r}}\Gamma\left(\frac{\bm{\rho}}{\textbf{r}}\right)
z^{-\tau(\bm{\rho},\textbf{r},\mathfrak{D})}
 \right)\dx z
 \\
 & +\frac{\left(-1\right)^{d}}{2\pi i}\sum_{\eta=0}^{h-1}\frac{\dbinom{h}{\eta}}{2^{\eta}}\sum_{\ell=0}^{\eta}\dbinom{\eta}{\ell}\pi^{\frac{h-\eta+\ell}{2}}\left(-1\right)^{\eta-\ell}\int_{\left(1/N\right)}e^{Nz}z^{-k-1-\frac{h-\eta+\ell}{2}}\left( 
 \sum_{\bm{\rho}\in Z^{d}}\frac1{\textbf{r}}\Gamma\left(\frac{\bm{\rho}}{\textbf{r}}\right)
z^{-\tau(\bm{\rho},\textbf{r},\mathfrak{D})}
 \right)\omega_{2}\left(\frac{\pi^{2}}{z}\right)^{h-\eta}\dx z\\
 & =:I_{2,1}+I_{2,2}.
\end{align*}

Let us consider $I_{2,1}$ which corresponds to $M_3$ in Theorem \ref{main_thm}. We want to show that it is possible to
exchange the integral with the product of the series involving the
non-trivial zeros if the Riemann Zeta function. To prove this, we
fix an arbitrary $1\leq\lambda\leq d$ and we analyze
\[
\sum_{\rho_{1}}\frac{\left|\Gamma\left(\frac{\rho_{1}}{r_{1}}\right)\right|}{r_{1}}\cdots\sum_{\rho_{\lambda}}\frac{\left|\Gamma\left(\frac{\rho_{\lambda}}{r_{\lambda}}\right)\right|}{r_{\lambda}}\int_{\left(1/N\right)}\left|e^{Nz}\right|\left|z\right|^{-k-1-\frac{\ell}{2}}\left|z^{-\tau\left(\bm{\rho},\mathbf{r},\mathfrak{J}_{\lambda}\right)}\right|\prod_{s=\lambda+1}^{d}\left|\sum_{\rho_{s}}\frac{\Gamma\left(\frac{\rho_{s}}{r_{s}}\right)}{r_{s}}z^{-\frac{\rho_{s}}{r_{s}}}\right|\left|\dx z\right|
\]
with the convention that, if $\lambda=d$, then $\prod_{s=\lambda+1}^{d}\left|\sum_{\rho_{s}}\frac{\Gamma\left(\frac{\rho_{s}}{r_{s}}\right)}{r_{s}}z^{-\frac{\rho_{s}}{r_{s}}}\right|=1$.
From Stirling formula (\ref{eq:Stirling}) and (\ref{eq:estserieszeros0})
we have that it is enough to study the convergence of
\begin{align*}
\sum_{\rho_{1}}\left|\gamma_{1}\right|^{\frac{\beta_{1}}{r_{1}}-\frac{1}{2}}\cdots\sum_{\rho_{\lambda}}\left|\gamma_{\lambda}\right|^{\frac{\beta_{\lambda}}{r_{\lambda}}-\frac{1}{2}}&\int_{\mathbb{R}}\left|z\right|^{-k-1-\frac{\ell}{2}-\tau\left(\bm{\beta},\mathbf{r},\mathfrak{J}_{\lambda}\right)}
\\
&\times\exp\left(\sum_{j=1}^{\lambda}\left(\frac{\gamma_{j}}{r_{j}}\arctan\left(Ny\right)-\frac{\pi\left|\gamma_{j}\right|}{2r_{j}}\right)\right)\prod_{s=\lambda+1}^{d}\left|\sum_{\rho_{s}}\frac{\Gamma\left(\frac{\rho_{s}}{r_{s}}\right)}{r_{s}}z^{-\frac{\rho_{s}}{r_{s}}}\right|\dx y.
\end{align*}

We split the integral in $\left|y\right|\leq1/N$ and $\left|y\right|>1/N$.
Assume that $\left|y\right|\leq1/N$, then, by (\ref{eq:estserieszeros0}),
we have
\begin{align*}
\sum_{\rho_{1}}\left|\gamma_{1}\right|^{\frac{\beta_{1}}{r_{1}}-\frac{1}{2}}\cdots\sum_{\rho_{\lambda}}\left|\gamma_{\lambda}\right|^{\frac{\beta_{\lambda}}{r_{\lambda}}-\frac{1}{2}}
&\int_{-1/N}^{1/N}\left|z\right|^{-k-1-\frac{\ell}{2}}\left|z\right|^{-\tau\left(\bm{\beta},\mathbf{r},\mathfrak{J}_{\lambda}\right)}
\exp\left(\sum_{j=1}^{\lambda}\left(\frac{\gamma_{j}}{r_{j}}\arctan\left(Ny\right)-\frac{\pi\left|\gamma_{j}\right|}{2r_{j}}\right)\right)N^{d-\lambda}\dx y
\\
&\ll_{N,\lambda,d}\sum_{\rho_{1}}\left|\gamma_{1}\right|^{\frac{\beta_{1}}{r_{1}}-\frac{1}{2}}\exp\left(-\frac{\pi\left|\gamma_{1}\right|}{4r_{1}}\right)\cdots\sum_{\rho_{\lambda}}\left|\gamma_{\lambda}\right|^{\frac{\beta_{\lambda}}{r_{\lambda}}-\frac{1}{2}}\exp\left(-\frac{\pi\left|\gamma_{\lambda}\right|}{4r_{\lambda}}\right)
\end{align*}
and the series trivially converges, so assume that $\left|y\right|>1/N.$
It is enough considering the case
\begin{align*}
\sum_{\rho_{1}}\left|\gamma_{1}\right|^{\frac{\beta_{1}}{r_{1}}-\frac{1}{2}}\cdots\sum_{\rho_{\lambda}}\left|\gamma_{\lambda}\right|^{\frac{\beta_{\lambda}}{r_{\lambda}}-\frac{1}{2}}
&\int_{\left|y\right|>1/N}\left|z\right|^{-k-1-\frac{\ell}{2}-\tau\left(\bm{\beta},\mathbf{r},\mathfrak{J}_{\lambda}\right)+\frac{\alpha}{2}}
\\
&\times\exp\left(\sum_{j=1}^{\lambda}\left(\frac{\gamma_{j}}{r_{j}}\arctan\left(Ny\right)-\frac{\pi\left|\gamma_{j}\right|}{2r_{j}}\right)\right)\log^{2\alpha}\left(2N\left|y\right|\right)\dx y
\end{align*}
for $1\leq\alpha\leq d-\lambda$, since the powers of $N$ do not
affect the study of the convergence and so can be omitted. Assume
$y>1/N$ and $\gamma_{j}>0,\,j=1,\dots,\lambda$. Putting $Ny=u$
and using the well-known identity $\arctan(x)-\frac{\pi}{2}=-\arctan\left(\frac{1}{x}\right)$
we get
\[
\sum_{\rho_{1}:\,\gamma_{1}>0}\gamma_{1}^{\frac{\beta_{1}}{r_{1}}-\frac{1}{2}}\cdots\sum_{\rho_{\lambda}:\,\gamma_{\lambda}>0}\gamma_{\lambda}^{\frac{\beta_{\lambda}}{r_{\lambda}}-\frac{1}{2}}
\int_{1}^{+\infty}u^{-k-1-\frac{\ell}{2}-\tau\left(\bm{\beta},\mathbf{r},\mathfrak{J}_{\lambda}\right)+\frac{\alpha}{2}}\exp\left(-\arctan\left(\frac{1}{u}\right)\tau\left(\bm{\gamma},\mathbf{r},\mathfrak{J}_{\lambda}\right)\right)\log^{2\alpha}\left(2u\right)\dx u
\]
and, by Lemma \ref{lem:generalized1}, we have the convergence if
$k>\frac{\lambda+\alpha-\ell}{2}$, and since this inequality must
be holds for all $0\leq\ell\leq d$ and all $1\leq\alpha\leq d-\lambda$
we can conclude that $k>\frac{d}{2}.$ Now, fix $1\leq\eta\leq\lambda$
and assume that $\gamma_{1},\dots,\gamma_{\eta}>0$ and $\gamma_{\eta+1},\dots,\gamma_{\lambda}<0$.
In this case, recalling that $y>1/N$ and so $\frac{\gamma_{j}}{r_{j}}\arctan\left(Ny\right)-\frac{\pi\left|\gamma_{j}\right|}{2r_{j}}\leq-\frac{\pi\left|\gamma_{j}\right|}{2r_{j}}$
for $j>\eta$, we have to work with
\begin{align*}
\sum_{\rho_{1}:\,\gamma_{1}>0}\gamma_{1}^{\frac{\beta_{1}}{r_{1}}-\frac{1}{2}}\cdots&\sum_{\rho_{\eta}:\,\gamma_{\eta}>0}\gamma_{\eta}^{\frac{\beta_{\eta}}{r_{\eta}}-\frac{1}{2}}\sum_{\rho_{\eta+1}:\,\gamma_{\eta+1}<0}\left|\gamma_{\eta+1}\right|^{\frac{\beta_{\eta+1}}{r_{\eta+1}}-\frac{1}{2}}\exp\left(-\frac{\pi\left|\gamma_{\eta+1}\right|}{2r_{\eta+1}}\right)\cdots\sum_{\rho_{\lambda}:\,\gamma_{\lambda}<0}\left|\gamma_{\lambda}\right|^{\frac{\beta_{\lambda}}{r_{\lambda}}-\frac{1}{2}}\exp\left(-\frac{\pi\left|\gamma_{\lambda}\right|}{2r_{\lambda}}\right)
\\
&\times\int_{y>1/N}y^{-k-1-\frac{\ell}{2}-\tau\left(\bm{\beta},\mathbf{r},\mathfrak{J}_{\lambda}\right)+\frac{\alpha}{2}}\exp\left(\sum_{j=1}^{\eta}\left(\frac{\gamma_{j}}{r_{j}}\arctan\left(Ny\right)-\frac{\pi\gamma_{j}}{2r_{j}}\right)\right)\log^{2\alpha}\left(2Ny\right)\dx y
\end{align*}
with $1\leq\alpha\leq d-\lambda$. Letting $Ny=u$, we note that we
have to deal with
\begin{align*}
\sum_{\rho_{1}:\,\gamma_{1}>0}\gamma_{1}^{\frac{\beta_{1}}{r_{1}}-\frac{1}{2}}\cdots&\sum_{\rho_{\eta}:\,\gamma_{\eta}>0}\gamma_{\eta}^{\frac{\beta_{\eta}}{r_{\eta}}-\frac{1}{2}}\sum_{\rho_{\eta+1}:\,\gamma_{\eta+1}<0}\left|\gamma_{\eta+1}\right|^{\frac{\beta_{\eta+1}}{r_{\eta+1}}-\frac{1}{2}}\exp\left(-\frac{\pi\left|\gamma_{\eta+1}\right|}{2r_{\eta+1}}\right)\cdots\sum_{\rho_{\lambda}:\,\gamma_{\lambda}<0}\left|\gamma_{\lambda}\right|^{\frac{\beta_{\lambda}}{r_{\lambda}}-\frac{1}{2}}\exp\left(-\frac{\pi\left|\gamma_{\lambda}\right|}{2r_{\lambda}}\right)
\\
&\times\int_{1}^{+\infty}u^{-k-1-\frac{\ell}{2}-\tau\left(\bm{\beta},\mathbf{r},\mathfrak{J}_{\lambda}\right)+\frac{\alpha}{2}}
\exp\left(-\arctan\left(\frac{1}{u}\right)\tau\left(\bm{\gamma},\mathbf{r},\mathfrak{J}_{\eta}\right)\right)\log^{2\alpha}\left(2u\right)\dx u
\end{align*}
(also in this case we omit the powers of $N$ because they do not
affect the convergence) and, since
\[
\left(\frac{1}{u}\right)^{\frac{\beta_{j}}{r_{j}}}<1,\,u>1,\,j=1,\dots,\lambda,
\]
it is enough to consider
\[
\sum_{\rho_{1}:\,\gamma_{1}>0}\gamma_{1}^{\frac{\beta_{1}}{r_{1}}-\frac{1}{2}}\cdots\sum_{\rho_{\eta}:\,\gamma_{\eta}>0}\gamma_{\eta}^{\frac{\beta_{\eta}}{r_{\eta}}-\frac{1}{2}}\int_{1}^{+\infty}u^{-k-1-\frac{\ell}{2}-\tau\left(\bm{\beta},\mathbf{r},\mathfrak{J}_{\eta}\right)+\frac{\alpha}{2}}\exp\left(-\arctan\left(\frac{1}{u}\right)\tau\left(\bm{\gamma},\mathbf{r},\mathfrak{J}_{\eta}\right)\right)\log^{2\alpha}\left(2u\right)\dx u
\]
and so, arguing as in the previous case, the convergence if $k>\frac{\eta+d-\lambda}{2}$
and so, since $\eta\leq\lambda$, a complete convergence in the case
$k>\frac{d}{2}$. If $y<-1/N$ we get the same bounds for $k$, by
symmetry. 

Now, since $\left|\mathfrak{D}\right|=d$, we have  
\begin{align*}
I_{2,1}=&\frac{\left(-1\right)^{d}}{2^{h+1}\pi i}\sum_{\ell=0}^{h}\dbinom{h}{\ell}\pi^{\frac{\ell}{2}}\left(-1\right)^{h-\ell}\sum_{\bm{\rho}\in Z^{d}}\frac1{\textbf{r}}\Gamma\left(\frac{\bm{\rho}}{\textbf{r}}\right)\int_{\left(1/N\right)}e^{Nz}z^{-k-1-\frac{\ell}{2}-\tau\left(\bm{\rho},\mathbf{r},\mathfrak{D}\right)}\dx z
\\
=&\frac{N^{k}\left(-1\right)^{d}}{2^{h}}\sum_{\ell=0}^{h}\dbinom{h}{\ell}\left(N\pi\right)^{\frac{\ell}{2}}\left(-1\right)^{h-\ell}\sum_{\bm{\rho}\in Z^{d}}
\frac1{\textbf{r}}\Gamma\left(\frac{\bm{\rho}}{\textbf{r}}\right)
\frac{N^{\tau\left(\bm{\rho},\mathbf{r},\mathfrak{D}\right)}}{\Gamma\left(k+1+\frac{\ell}{2}+\tau\left(\bm{\rho},\mathbf{r},\mathfrak{D}\right)\right)},
\end{align*}
from \eqref{eq:onemaintool}.

Now, we analyze $I_{2,2}$ (which corresponds to $M_4$ in Theorem \ref{main_thm}) and we prove that we can switch the integral
with the series involving the non-trivial zeros of the Riemann Zeta
function and with the powers of $\omega_{2}\left(\frac{\pi^{2}}{z}\right)$.
As the previous calculations, we fix $1\leq\lambda\leq d$ and $1\leq\alpha\leq h-\eta$.
So we have to consider 
\begin{align*}
\sum_{\rho_{1}}\frac{\left|\Gamma\left(\frac{\rho_{1}}{r_{1}}\right)\right|}{r_{1}}\cdots\sum_{\rho_{\lambda}}\frac{\left|\Gamma\left(\frac{\rho_{\lambda}}{r_{\lambda}}\right)\right|}{r_{\lambda}}
&\sum_{\mathbf{f}\in\left(\mathbb{N}^{+}\right)^{\alpha}}\int_{\left(1/N\right)}\left|e^{Nz}\right|e^{-\mathrm{Re}\left(\frac{\pi^{2}}{z}\right)\left\Vert \mathbf{f}\right\Vert ^{2}}\left|z\right|^{-k-1-\frac{h-\eta+\ell}{2}}\left|z^{-\tau\left(\bm{\rho},\mathbf{r},\mathfrak{J}_{\lambda}\right)}\right|
\\
&\times\prod_{s=\lambda+1}^{d}\left|\sum_{\rho_{s}}\frac{\Gamma\left(\frac{\rho_{s}}{r_{s}}\right)}{r_{s}}z^{-\frac{\rho_{s}}{r_{s}}}\right|\left|\omega_{2}\left(\frac{\pi^{2}}{z}\right)\right|^{h-\eta-\alpha}\left|\dx z\right|
\end{align*}
again with the convention $\prod_{s=\lambda+1}^{d}\left|\sum_{\rho_{s}}\frac{\Gamma\left(\frac{\rho_{s}}{r_{s}}\right)}{r_{s}}z^{-\frac{\rho_{s}}{r_{s}}}\right|=1$
if $\lambda=d$. From (\ref{eq:Stirling}) and recalling that the powers of $N$ do not affect the convergence, it is enough to study the convergence of
\begin{align*}
\sum_{\rho_{1}}\left|\gamma_{1}\right|^{\frac{\beta_{1}}{r_{1}}-\frac{1}{2}}\cdots&\sum_{\rho_{\lambda}}\left|\gamma_{\lambda}\right|^{\frac{\beta_{\lambda}}{r_{\lambda}}-\frac{1}{2}}
\sum_{\mathbf{f}\in\left(\mathbb{N}^{+}\right)^{\alpha}}\int_{\mathbb{R}}\left|z\right|^{-k-1-\frac{h-\eta+\ell}{2}}\left|z\right|^{-\tau\left(\bm{\beta},\mathbf{r},\mathfrak{J}_{\lambda}\right)}e^{-\frac{\left\Vert \mathbf{f}\right\Vert ^{2}N}{1+N^{2}y^{2}}}
\\
&\times\exp\left(\sum_{j=1}^{\lambda}\left(\frac{\gamma_{j}}{r_{j}}\arctan\left(Ny\right)-\frac{\pi\left|\gamma_{j}\right|}{2r_{j}}\right)\right)\prod_{s=\lambda+1}^{d}\left|\sum_{\rho_{s}}\frac{\Gamma\left(\frac{\rho_{s}}{r_{s}}\right)}{r_{s}}z^{-\frac{\rho_{s}}{r_{s}}}\right|\left(\frac{1+y^{2}N^{2}}{N}\right)^{\frac{h-\eta-\alpha}{2}}\dx y.
\end{align*}
If $\left|y\right|\leq1/N$ we have
\begin{align*}
\sum_{\rho_{1}}\left|\gamma_{1}\right|^{\frac{\beta_{1}}{r_{1}}-\frac{1}{2}}\cdots&\sum_{\rho_{\lambda}}\left|\gamma_{\lambda}\right|^{\frac{\beta_{\lambda}}{r_{\lambda}}-\frac{1}{2}}\sum_{\mathbf{f}\in\left(\mathbb{N}^{+}\right)^{\alpha}}\int_{\mathbb{-}1/N}^{1/N}\left|z\right|^{-k-1-\frac{h-\eta+\ell}{2}}\left|z\right|^{-\tau\left(\bm{\beta},\mathbf{r},\mathfrak{J}_{\lambda}\right)}e^{-\frac{\left\Vert \mathbf{f}\right\Vert ^{2}N}{1+N^{2}y^{2}}}
\\
&\times\exp\left(\sum_{j=1}^{\lambda}\left(\frac{\gamma_{j}}{r_{j}}\arctan\left(Ny\right)-\frac{\pi\left|\gamma_{j}\right|}{2r_{j}}\right)\right)\prod_{s=\lambda+1}^{d}\left|\sum_{\rho_{s}}\frac{\Gamma\left(\frac{\rho_{s}}{r_{s}}\right)}{r_{s}}z^{-\frac{\rho_{s}}{r_{s}}}\right|\left(\frac{1+y^{2}N^{2}}{N}\right)^{\frac{h-\eta-\alpha}{2}}\dx y
\\
\ll_{N,k,\alpha,h,\eta,\ell,\mathbf{r}}&\sum_{\rho_{1}}\left|\gamma_{1}\right|^{\frac{\beta_{1}}{r_{1}}-\frac{1}{2}}\exp\left(-\frac{\pi\left|\gamma_{1}\right|}{4r_{1}}\right)\cdots\sum_{\rho_{\lambda}}\left|\gamma_{\lambda}\right|^{\frac{\beta_{\lambda}}{r_{\lambda}}-\frac{1}{2}}\exp\left(-\frac{\pi\left|\gamma_{1}\right|}{4r_{1}}\right)\sum_{\mathbf{f}\in\left(\mathbb{N}^{+}\right)^{\alpha}}e^{-\left\Vert \mathbf{f}\right\Vert ^{2}N}
\end{align*}
and trivially the convergence, so we consider now $\left|y\right|>1/N$.
If we fix $1\leq\mu\leq d-\lambda$, from (\ref{eq:estserieszeros0}),
it is enough to work with
\begin{align*}
\sum_{\rho_{1}}\left|\gamma_{1}\right|^{\frac{\beta_{1}}{r_{1}}-\frac{1}{2}}\cdots&\sum_{\rho_{\lambda}}\left|\gamma_{\lambda}\right|^{\frac{\beta_{\lambda}}{r_{\lambda}}-\frac{1}{2}}
\sum_{\mathbf{f}\in\left(\mathbb{N}^{+}\right)^{\alpha}}
\int_{\left|y\right|>1/N}\left|z\right|^{-k-1-\frac{h-\eta+\ell}{2}+\frac{\mu}{2}}\left|z\right|^{-\tau\left(\bm{\beta},\mathbf{r},\mathfrak{J}_{\lambda}\right)}e^{-\frac{\left\Vert \mathbf{f}\right\Vert ^{2}}{Ny^{2}}}
\\
&\times\exp\left(\sum_{j=1}^{\lambda}\left(\frac{\gamma_{j}}{r_{j}}\arctan\left(Ny\right)-\frac{\pi\left|\gamma_{j}\right|}{2r_{j}}\right)\right)\log^{2\mu}\left(2N\left|y\right|\right)\left|y\right|^{h-\eta-\alpha}\dx y.
\end{align*}
Assume $y>1/N$ and $\gamma_{j}>0,\,j=1,\dots,\lambda$. Since $\arctan\left(\frac{1}{Ny}\right)\gg\frac{1}{Ny}$,
We have 
\begin{align*}
\sum_{\rho_{1}:\,\gamma_{1}>0}\gamma_{1}^{\frac{\beta_{1}}{r_{1}}-\frac{1}{2}}\cdots&\sum_{\rho_{\lambda}:\,\gamma_{\lambda}>0}\gamma_{\lambda}^{\frac{\beta_{\lambda}}{r_{\lambda}}-\frac{1}{2}}
\sum_{\mathbf{f}\in\left(\mathbb{N}^{+}\right)^{\alpha}}
\int_{1/N}^{+\infty}y^{-k-1-\frac{h-\eta+\ell}{2}-\tau\left(\bm{\beta},\mathbf{r},\mathfrak{J}_{\lambda}\right)+\frac{\mu}{2}}e^{-\frac{\left\Vert \mathbf{f}\right\Vert ^{2}}{Ny^{2}}}
\\
&\times\exp\left(-\frac{\tau\left(\bm{\gamma},\mathbf{r},\mathfrak{J}_{\lambda}\right)}{Ny}\right)\log^{2\mu}\left(2Ny\right)y^{h-\eta-\alpha}\dx y.
\end{align*}

Putting $v=\frac{\sum_{j=1}^{\lambda}\frac{\gamma_{j}}{r_{j}}}{Ny},$
we obtain
\begin{align*}
\frac{\sum_{\rho_{1}:\,\gamma_{1}>0}\gamma_{1}^{\frac{\beta_{1}}{r_{1}}-\frac{1}{2}}\cdots\sum_{\rho_{\lambda}:\,\gamma_{\lambda}>0}\gamma_{\lambda}^{\frac{\beta_{\lambda}}{r_{\lambda}}-\frac{1}{2}}}{\left(\tau\left(\bm{\gamma},\mathbf{r},\mathfrak{J}_{\lambda}\right)\right)^{k+\frac{h-\eta+\ell}{2}+\tau\left(\bm{\beta},\mathbf{r},\mathfrak{J}_{\lambda}\right)-h+\eta+\alpha-\frac{\mu}{2}}}
&\sum_{\mathbf{f}\in\left(\mathbb{N}^{+}\right)^{\alpha}}
\int_{0}^{\tau\left(\bm{\gamma},\mathbf{r},\mathfrak{J}_{\lambda}\right)}v^{k-1+\frac{h-\eta+\ell}{2}+\tau\left(\bm{\beta},\mathbf{r},\mathfrak{J}_{\lambda}\right)-\frac{\mu}{2}}
\\
&\times e^{-\frac{\left\Vert \mathbf{f}\right\Vert ^{2}Nv^{2}}{\tau\left(\bm{\beta},\mathbf{r},\mathfrak{J}_{\gamma}\right)^{2}}}\exp\left(-v\right)\log^{2\mu}\left(2\frac{\tau\left(\bm{\gamma},\mathbf{r},\mathfrak{J}_{\lambda}\right)}{v}\right)v^{-h+\eta+\alpha}\dx v.
\end{align*}
Now, from (\ref{eq:stima semplice}) and by elementary manipulations
we can study
\begin{align*}
\frac{\sum_{\rho_{1}:\,\gamma_{1}>0}\gamma_{1}^{-\frac{1}{2}}\cdots\sum_{\rho_{\lambda}:\,\gamma_{\lambda}>0}\gamma_{\lambda}^{-\frac{1}{2}}}{\left(\tau\left(\bm{\gamma},\mathbf{r},\mathfrak{J}_{\lambda}\right)\right)^{k-\frac{h-\eta}{2}+\frac{\ell}{2}+\alpha-\frac{\mu}{2}}}
&\sum_{\mathbf{f}\in\left(\mathbb{N}^{+}\right)^{\alpha}}
\int_{0}^{\tau\left(\bm{\gamma},\mathbf{r},\mathfrak{J}_{\lambda}\right)}v^{k-1-\frac{h-\eta}{2}+\frac{\ell}{2}+\tau\left(\bm{\beta},\mathbf{r},\mathfrak{J}_{\lambda}\right)+\alpha-\frac{\mu}{2}}
\\
&\times e^{-\frac{\left\Vert \mathbf{f}\right\Vert ^{2}Nv^{2}}{\tau\left(\bm{\gamma},\mathbf{r},\mathfrak{J}_{\lambda}\right)^{2}}}\exp\left(-v\right)\log^{2\nu}\left(2\frac{\tau\left(\bm{\gamma},\mathbf{r},\mathfrak{J}_{\lambda}\right)}{v}\right)\dx v
\end{align*}
and so, by Lemma \ref{lem:generalizlemnew}, we have the convergence
if $k>\frac{\lambda}{2}+\frac{h-\eta-\ell+\mu}{2}.$ Since $\mu\leq d-\lambda$
and $0\leq\ell\leq\eta$ we have the complete convergence for all
possible cases if $k>\frac{d+h}{2}.$ 

Now fix $1\leq\xi\leq\lambda$ and assume that $\gamma_{1},\dots,\gamma_{\xi}>0$
and $\gamma_{\xi+1},\dots,\gamma_{\lambda}<0$. In this case, recalling
that $y>1/N,$we have to work with
\begin{align*}
&\sum_{\rho_{1}:\,\gamma_{1}>0}\gamma_{1}^{\frac{\beta_{1}}{r_{1}}-\frac{1}{2}}\cdots\sum_{\rho_{\xi}:\,\gamma_{\xi}>0}\gamma_{\xi}^{\frac{\beta_{\xi}}{r_{\xi}}-\frac{1}{2}}\sum_{\rho_{\xi+1}:\,\gamma_{\xi+1}<0}\left|\gamma_{\xi+1}\right|^{\frac{\beta_{\xi+1}}{r_{\xi+1}}-\frac{1}{2}}\cdots\sum_{\rho_{\lambda}:\,\gamma_{\lambda}<0}\left|\gamma_{\lambda}\right|^{\frac{\beta_{\lambda}}{r_{\lambda}}-\frac{1}{2}}
\\
&\times\sum_{\mathbf{m}\in\left(\mathbb{N}^{+}\right)^{\alpha}}
\int_{1/N}^{+\infty}y^{-k-1-\frac{h-\eta+\ell}{2}-\tau\left(\bm{\beta},\mathbf{r},\mathfrak{J}_{\lambda}\right)+\frac{\mu}{2}}e^{-\frac{\left\Vert \mathbf{m}\right\Vert ^{2}}{Ny^{2}}}
\exp\left(\sum_{j=1}^{\lambda}\left(\frac{\gamma_{j}}{r_{j}}\arctan\left(Ny\right)-\frac{\pi\left|\gamma_{j}\right|}{2r_{j}}\right)\right)\log^{2\mu}\left(2Ny\right)y^{h-\eta-\alpha}\dx y.
\end{align*}
Now, since $y>1/N$, if $\gamma_{j}<0$, we observe that $\frac{\gamma_{j}}{r_{j}}\arctan\left(Ny\right)-\frac{\pi\left|\gamma_{j}\right|}{2r_{j}}\leq-\frac{\pi\left|\gamma_{j}\right|}{2r_{j}}$,
$y^{-\beta_{j}/r_{j}}\leq N^{\beta_{j}/r_{j}}\leq N^{1/r_{j}}$ and
\[
\sum_{\rho_{j}:\,\gamma_{j}<0}\left|\gamma_{j}\right|^{\frac{\beta_{j}}{r_{j}}-\frac{1}{2}}\exp\left(-\frac{\pi\left|\gamma_{j}\right|}{2r_{j}}\right)
\]
trivially converges, so it is enough to consider
\begin{align*}
\sum_{\rho_{1}:\,\gamma_{1}>0}\gamma_{1}^{\frac{\beta_{1}}{r_{1}}-\frac{1}{2}}&\cdots\sum_{\rho_{\xi}:\,\gamma_{\xi}>0}\gamma_{\xi}^{\frac{\beta_{\xi}}{r_{\xi}}-\frac{1}{2}}
\sum_{\mathbf{f}\in\left(\mathbb{N}^{+}\right)^{\alpha}}
\int_{1/N}^{+\infty}y^{-k-1-\frac{h-\eta+\ell}{2}-\tau\left(\bm{\beta},\mathbf{r},\mathfrak{J}_{\xi}\right)+\frac{\mu}{2}}e^{-\frac{\left\Vert \mathbf{f}\right\Vert ^{2}}{Ny^{2}}}
\\
&\times\exp\left(\sum_{j=1}^{\xi}\left(\frac{\gamma_{j}}{r_{j}}\arctan\left(Ny\right)-\frac{\pi\gamma_{j}}{2r_{j}}\right)\right)\log^{2\mu}\left(2Ny\right)y^{h-\eta-\alpha}\dx y
\end{align*}
and so, following the previous case, we have the convergence if $k>\frac{h+d-\lambda+\xi}{2}$
and since $\xi\leq\lambda$, we have the complete convergence if $k>\frac{d+h}{2}$.
If $y<-1/N$ we have the same bounds, by the symmetry of the non-trivial
zeros of the Riemann Zeta function, so we can finally exchange the
integral with the series and get 
\[
I_{2,2}=\frac{\left(-1\right)^{d}}{2\pi i}\sum_{\eta=0}^{h-1}\frac{\dbinom{h}{\eta}}{2^{\eta}}\sum_{\ell=0}^{\eta}\dbinom{\eta}{\ell}\pi^{\frac{h-\eta+\ell}{2}}\left(-1\right)^{\eta-\ell}
 \sum_{\bm{\rho}\in Z^{d}}\frac1{\textbf{r}}\Gamma\left(\frac{\bm{\rho}}{\textbf{r}}\right)
\sum_{\mathbf{f}\in\left(\mathbb{N}^{+}\right)^{h-\eta}}
\int_{\left(1/N\right)}e^{Nz}z^{-k-1-\frac{h-\eta+\ell}{2}-\tau\left(\bm{\rho},\mathbf{r},\mathfrak{D}\right)}e^{-\frac{\pi^{2}\left\Vert \mathbf{f}\right\Vert ^{2}}{z}}\dx z
\]
which is, taking $Nz=v$, 
\[
\frac{\left(-1\right)^{d}}{2\pi i}\sum_{\eta=0}^{h-1}\frac{\dbinom{h}{\eta}}{2^{\eta}}\sum_{\ell=0}^{\eta}\dbinom{\eta}{\ell}\pi^{\frac{h-\eta+\ell}{2}}\left(-1\right)^{\eta-\ell}
 \sum_{\bm{\rho}\in Z^{d}}\frac1{\textbf{r}}\Gamma\left(\frac{\bm{\rho}}{\textbf{r}}\right)
N^{k+\frac{h-\eta+\ell}{2}+\tau\left(\bm{\rho},\mathbf{r},\mathfrak{D}\right)}\sum_{\mathbf{f}\in\left(\mathbb{N}^{+}\right)^{h-\eta}}\int_{\left(1\right)}z^{-k-1-\frac{h-\eta+\ell}{2}-\tau\left(\bm{\rho},\mathbf{r},\mathfrak{D}\right)}e^{v-\frac{\pi^{2}N\left\Vert \mathbf{f}\right\Vert }{v}}\dx v
\]
\[
=\frac{N^{k/2}\left(-1\right)^{d}}{\pi^{k}}\sum_{\eta=0}^{h-1}\frac{\dbinom{h}{\eta}}{2^{\eta}}\sum_{\ell=0}^{\eta}\dbinom{\eta}{\ell}\left(-1\right)^{\eta-\ell}
 \sum_{\bm{\rho}\in Z^{d}}\frac1{\textbf{r}}\Gamma\left(\frac{\bm{\rho}}{\textbf{r}}\right)
\frac{N^{\frac{h-\eta+\ell}{4}+\tau\left(\bm{\rho},\mathbf{r},\mathfrak{D}\right)/2}}{\pi^{\tau\left(\bm{\rho},\mathbf{r},\mathfrak{D}\right)}}
\sum_{\mathbf{f}\in\left(\mathbb{N}^{+}\right)^{h-\eta}}\frac{J_{k+\frac{h-\eta+\ell}{2}+\tau\left(\bm{\rho},\mathbf{r},\mathfrak{D}\right)}\left(2\pi\sqrt{N}\left\Vert \mathbf{f}\right\Vert \right)}{\left\Vert \mathbf{f}\right\Vert ^{k+\frac{h-\eta+\ell}{2}+\tau\left(\bm{\rho},\mathbf{r},\mathfrak{D}\right)}}.
\]

\section{Evaluation of $I_{3}$}

In this section we evaluate 
\[
I_{3}=\frac{1}{2\pi i}\sum_{\underset{{\scriptstyle \left|I\right|\geq1}}{I\subseteq\mathfrak{D}}}\left(-1\right)^{\left|\mathfrak{D}\setminus I\right|}\int_{\left(1/N\right)}e^{Nz}z^{-k-1-\tau\left(\mathbf{r},I\right)} \left(
 \sum_{\bm{\rho}\in Z^{|\mathfrak{D}\setminus I|}}
\frac1{\textbf{r}}\Gamma\left(\frac{\bm{\rho}}{\textbf{r}}\right)
z^{-\tau\left(\bm{\rho},\mathbf{r},\mathfrak{D}\setminus I\right)}
\right)\omega_{2}\left(z\right)^{h}\dx z
\]
which corresponds to $M_5$ in Theorem \ref{main_thm} and has a similar structure to $I_{2}$. If we fix $I\subseteq\mathfrak{D}$
we can repeat the previous argument to justify the exchange the integral
with the series only considering $k+\tau\left(\mathbf{r},I\right)$
instead of $k$ and $\left|\mathfrak{D}\setminus I\right|$ instead
of $d=\left|\mathfrak{D}\right|$. So, we can conclude that all exchanges
are justified if $k>\frac{\left|\mathfrak{D}\setminus I\right|+h}{2}-\tau\left(\mathbf{r},I\right)$,
and so we have
\begin{align*}
I_{3}=&\frac{1}{2\pi i}\sum_{\underset{{\scriptstyle \left|I\right|\geq1}}{I\subseteq\mathfrak{D}}}\left(-1\right)^{\left|\mathfrak{D}\setminus I\right|}\sum_{\eta=0}^{h}\frac{\dbinom{h}{\eta}}{2^{\eta}}\sum_{\ell=0}^{\eta}\dbinom{\eta}{\ell}\left(-1\right)^{\eta-\ell}\pi^{\frac{h-\eta+\ell}{2}}
\\
&\times
 \sum_{\bm{\rho}\in Z^{|\mathfrak{D}\setminus I|}}
\frac1{\textbf{r}}\Gamma\left(\frac{\bm{\rho}}{\textbf{r}}\right)
\int_{\left(1/N\right)}e^{Nz}z^{-k-1-\tau\left(\mathbf{r},I\right)-\frac{h-\eta+\ell}{2}-\tau\left(\bm{\rho},\mathbf{r},\mathfrak{D}\right)}\omega_{2}\left(\frac{\pi^{2}}{z}\right)^{h-\eta}\dx z
\\
=&\frac{1}{2\pi i}\sum_{\underset{{\scriptstyle \left|I\right|\geq1}}{I\subseteq\mathfrak{D}}}\left(-1\right)^{\left|\mathfrak{D}\setminus I\right|}\sum_{\eta=0}^{h}\frac{\dbinom{h}{\eta}}{2^{\eta}}\sum_{\ell=0}^{\eta}\dbinom{\eta}{\ell}\left(-1\right)^{\eta-\ell}\pi^{\frac{h-\eta+\ell}{2}}
\\
&\times
 \sum_{\bm{\rho}\in Z^{|\mathfrak{D}\setminus I|}}
\frac1{\textbf{r}}\Gamma\left(\frac{\bm{\rho}}{\textbf{r}}\right)
\sum_{\mathbf{f}\in\left(\mathbb{N}^{+}\right)^{h-\eta}}\int_{\left(1/N\right)}e^{Nz}z^{-k-1-\tau\left(\mathbf{r},I\right)-\frac{h-\eta+\ell}{2}-\tau\left(\bm{\rho},\mathbf{r},\mathfrak{D}\right)}e^{-\frac{\pi^{2}\left\Vert \mathbf{f}\right\Vert ^{2}}{z}}\dx z
\end{align*}
and so, taking $Nz=v$ and using (\ref{eq:bessel}), we have that
\begin{align*}
I_{3}=&\frac{N^{k}}{2\pi i}\sum_{\underset{{\scriptstyle \left|I\right|\geq1}}{I\subseteq\mathfrak{D}}}N^{\tau\left(\mathbf{r},I\right)}\left(-1\right)^{\left|\mathfrak{D}\setminus I\right|}\sum_{\eta=0}^{h}\frac{\dbinom{h}{\eta}}{2^{\eta}}\sum_{\ell=0}^{\eta}\dbinom{\eta}{\ell}\left(-1\right)^{\eta-\ell}\pi^{\frac{h-\eta+\ell}{2}}
\\
&\times
 \sum_{\bm{\rho}\in Z^{|\mathfrak{D}\setminus I|}}
\frac1{\textbf{r}}\Gamma\left(\frac{\bm{\rho}}{\textbf{r}}\right)
N^{\frac{h-\eta+\ell}{2}-\tau\left(\bm{\rho},\mathbf{r},\mathfrak{D}\setminus I\right)}\sum_{\mathbf{f}\in\left(\mathbb{N}^{+}\right)^{h-\eta}}\int_{\left(1\right)}e^{v}v^{-k-1-\tau\left(\mathbf{r},I\right)-\frac{h-\eta+\ell}{2}-\tau\left(\bm{\rho},\mathbf{r},\mathfrak{D}\right)}e^{-\frac{\pi^{2}N\left\Vert \mathbf{f}\right\Vert ^{2}}{v}}\dx v
\\
=&\frac{N^{k/2}}{\pi^{k}}\sum_{\underset{{\scriptstyle \left|I\right|\geq1}}{I\subseteq\mathfrak{D}}}N^{\tau\left(\mathbf{r},I\right)/2}\left(-1\right)^{\left|\mathfrak{D}\setminus I\right|}\sum_{\eta=0}^{h}\frac{\dbinom{h}{\eta}}{2^{\eta}}\sum_{\ell=0}^{\eta}\dbinom{\eta}{\ell}\left(-1\right)^{\eta-\ell}
\\
&\times
 \sum_{\bm{\rho}\in Z^{|\mathfrak{D}\setminus I|}}
\frac1{\textbf{r}}\Gamma\left(\frac{\bm{\rho}}{\textbf{r}}\right)
\frac{N^{\frac{h-\eta+\ell}{4}+\tau\left(\bm{\rho},\mathbf{r},\mathfrak{D}\setminus I\right)/2}}{\pi^{\tau\left(\bm{\rho},\mathbf{r},\mathfrak{D}\setminus I\right)}}
\sum_{\mathbf{f}\in\left(\mathbb{N}^{+}\right)^{h-\eta}}\frac{J_{k+\tau\left(\mathbf{r},I\right)+\frac{h-\eta+\ell}{2}+\tau\left(\bm{\rho},\mathbf{r},\mathfrak{D}\setminus I\right)}\left(2\pi\sqrt{N}\left\Vert \mathbf{f}\right\Vert \right)}{\left\Vert \mathbf{f}\right\Vert ^{k+\tau\left(\mathbf{r},I\right)+\frac{h-\eta+\ell}{2}+\tau\left(\bm{\rho},\mathbf{r},\mathfrak{D}\setminus I\right)}}
\end{align*}
and this completes the proof.

\bigskip
Marco Cantarini

Dipartimento di Ingegneria Industriale e Scienze Matematiche,

Universit\`a Politecnica delle Marche

Via delle Brecce Bianche, 12 

06131 Ancona, Italy

email (MC): m.cantarini@univpm.it

$\ $

Alessandro Gambini 

Dipartimento di Matematica Guido Castelnuovo,

Sapienza Universit\`a di Roma 

Piazzale Aldo Moro, 5 

00185 Roma, Italy 

email (AG): alessandro.gambini@uniroma1.it

$\ $

Alessandro Zaccagnini 

Dipartimento di Scienze, Matematiche, Fisiche e Informatiche,

Universit\`a di Parma 

Parco Area delle Scienze, 53/a 

43124 Parma, Italy 

email (AZ): alessandro.zaccagnini@unipr.it
\end{document}